\documentclass[11pt,a4paper]{amsart}
\usepackage{amsmath,amssymb, amsbsy}
\usepackage{color,psfrag}
\usepackage[dvips]{graphicx}
\usepackage{enumerate}
\newcommand{\R}{{\mathbb R}}
\newcommand{\N}{{\mathbb N}}

\newcommand{\lv}{\left\vert\kern-0.25ex\left\vert\kern-0.25ex\left\vert}
\newcommand{\rv}{\right\vert\kern-0.25ex\right\vert\kern-0.25ex\right\vert}

\renewcommand{\ge }{\geqslant}
\renewcommand{\geq }{\geqslant}
\renewcommand{\le }{\leqslant}
\renewcommand{\leq }{\leqslant}

\def\neweq#1{\begin{equation}\label{#1}}
\def\endeq{\end{equation}}
\def\eq#1{(\ref{#1})}
\newtheorem{theorem}{Theorem}[section]
\newtheorem{proposition}[theorem]{Proposition}
\newtheorem{lemma}[theorem]{Lemma}
\newtheorem{corollary}[theorem]{Corollary}

\newtheorem{remark}[theorem]{Remark}
\newtheorem{definition}[theorem]{Definition}
\theoremstyle{definition}

\begin{document}

\title[nonhomogeneous plates]{On the first frequency of reinforced \\ partially hinged plates}

\author[Elvise BERCHIO]{Elvise BERCHIO}
\address{\hbox{\parbox{5.7in}{\medskip\noindent{Dipartimento di Scienze Matematiche, \\
Politecnico di Torino,\\ Corso Duca degli Abruzzi 24, 10129 Torino, Italy. \\[3pt]
\em{E-mail address: }{\tt elvise.berchio@polito.it}}}}}
\author[Alessio FALOCCHI]{Alessio FALOCCHI}
\address{\hbox{\parbox{5.7in}{\medskip\noindent{Dipartimento di Scienze Matematiche, \\
Politecnico di Torino,\\ Corso Duca degli Abruzzi 24, 10129 Torino, Italy. \\[3pt]
\em{E-mail address: }{\tt alessio.falocchi@polito.it}}}}}
\author[Alberto FERRERO]{Alberto FERRERO}
\address{\hbox{\parbox{5.7in}{\medskip\noindent{ Dipartimento di Scienze e Innovazione Tecnologica, \\
Universit\'a del Piemonte Orientale,\\       Viale Teresa Michel 11, Alessandria, 15121, Italy. \\[3pt]
 \em{E-mail address: }{\tt alberto.ferrero@uniupo.it }}}}}
\author[Debdip GANGULY]{Debdip GANGULY}
\address{\hbox{\parbox{5.7in}{\medskip\noindent{Indian Institute of Science Education and Research,\\
Dr. Homi Bhabha Road, Pashan, Pune 411008, India. \\[3pt]       \em{E-mail addresses: }{\tt debdipmath@gmail.com}}}}}

\date{\today}



\begin{abstract} We consider a partially hinged rectangular plate and its normal modes. The dynamical properties of the plate are influenced by the spectrum
of the associated eigenvalue problem. In order to improve the stability of the plate, we place a certain amount of denser material in appropriate regions. If we look at the partial differential equation appearing in the model, this corresponds to insert a suitable
weight coefficient inside the equation. A possible way to locate such regions is to study the eigenvalue problem associated to the aforementioned weighted equation. In the present paper we focus our attention essentially on the first eigenvalue and on its minimization in terms of the weight. We prove the existence of minimizing weights inside special classes and we try to describe them together with the corresponding eigenfunctions.
\end{abstract}

\maketitle

\section{Introduction}\label{1}

Following \cite{fergaz} one may view a bridge as a long narrow rectangular thin plate $\Omega$ hinged at two opposite edges and free on the remaining
two edges: this plate well describes decks of footbridges and suspension bridges which, at the short edges, are supported by the ground. We refer to the monograph \cite{bookgaz} for a detailed survey of old and new mathematical models for suspension bridges. Up to scaling,
we may assume that the plate has length $\pi$ and width $2\ell$ with $2\ell\ll\pi$ so that
$$
\Omega=(0,\pi)\times(-\ell,\ell)\subset\R^2\, .
$$
There is a growing interest of engineers on the shape optimization for the design of bridges and, in particular, on the sensitivity analysis of certain
eigenvalue problems, see \cite[Chapter 6]{jhnm}. As pointed out by Banerjee \cite{banerjee}, {\em the free vibration analysis is a fundamental
pre-requisite before carrying out a flutter analysis}. Whence, in the the stability analysis of the plate a central role is played by the following eigenvalue problem:
\begin{equation}\label{weight0}
\begin{cases}
\Delta^2 u=\lambda\, u & \qquad \text{in } \Omega \\
u(0,y)=u_{xx}(0,y)=u(\pi,y)=u_{xx}(\pi,y)=0 & \qquad \text{for } y\in (-\ell,\ell) \\
u_{yy}(x,\pm\ell)+\sigma
u_{xx}(x,\pm\ell)=u_{yyy}(x,\pm\ell)+(2-\sigma)u_{xxy}(x,\pm\ell)=0
& \qquad \text{for } x\in (0,\pi)\, ,
\end{cases}
\end{equation}
where $\sigma$ denotes the Poisson ratio of the material forming the plate.
Throughout the paper we consider $\sigma\in(0,1/2)$, a range of values that includes most of the elastic materials. The boundary conditions on the short
edges tell that the plate is hinged; these conditions are attributed to Navier, since their first appearance in \cite{navier}. We refer to  \cite{bebuga2}
for the derivation of \eqref{weight0} from the total energy of the plate. Note that in \cite{fergaz} the whole spectrum of \eq{weight0} was determined,
while in \cite{bfg} the results were exploited to study the so-called torsional stability of suspension bridges for small energies.
Furthermore, in \cite{bebuga2} the variation of the eigenvalues, under domain deformations, which may not preserve the area, was investigated, see also \cite{lamberti} for related results about Dirichlet polyharmonic eigenvalue problems.

In the engineering literature the critical threshold for the wind velocity at which a form of dynamical instability, named flutter, arises, is commonly related to the distance between the square of the frequencies of certain oscillating modes. We refer to \cite{bebuga2} for a discussion of possible formulas to compute the above mentioned threshold. In particular, it follows that a possible way to increase this threshold is by increasing the distance between eigenvalues. Having this goal in mind, in order to improve the stability of the plate, we study the effect of inserting a denser material within it. This can be modelled in mathematical terms by a suitable weight function $p$; for this reason we study the \emph{weighted }eigenvalue problem:
\begin{equation}\label{weight}
\begin{cases}
\Delta^2 u=\lambda\, p(x,y) u & \qquad \text{in } \Omega \\
u(0,y)=u_{xx}(0,y)=u(\pi,y)=u_{xx}(\pi,y)=0 & \qquad \text{for } y\in (-\ell,\ell)\\
u_{yy}(x,\pm\ell)+\sigma
u_{xx}(x,\pm\ell)=u_{yyy}(x,\pm\ell)+(2-\sigma)u_{xxy}(x,\pm\ell)=0
& \qquad \text{for } x\in (0,\pi)\, ,
\end{cases}
\end{equation}
where, for $\alpha,\beta \in(0,+\infty)$ with $\alpha<\beta$ fixed, $p$ belongs to the following family of weights:

\begin{equation} \label{eq:famiglia}
P_{\alpha, \beta}:=\left\{p\in L^\infty(\Omega): \alpha\leq p\leq\beta \ \text{a.e. in } \Omega \
\text{ and }\int_{\Omega}p\,dxdy=|\Omega| \, \right\} \, .
\end{equation}

The spectral analysis of \eq{weight} should indicate where to place the denser material within the plate. In this respect, the condition on the integral of $p$ is posed in order to make the comparison with the case $p\equiv 1$ consistent. It is worth mentioning that a related linear problem has been recently treated in \cite{bbgza}, by studying the equation
$$
\Delta^2 u=\frac{f(x,y)}{1+d\chi_D(x,y)}\,\quad  \text{in } \Omega
$$
subject to the boundary conditions in \eq{weight}, where $\chi_D$ is the characteristic function of $D \subset \Omega$ and $d$ is a positive constant. The solution $u$ of this equation describes the
vertical displacement of the plate under the action of a load $f$ and the weight $1/(1+d\chi_D(x,y))$ is seen as an aerodynamic damper placed in $D$ in order to reduce the action of the external force.
The spectral analysis of \eq{weight} can help to complete and enrich the results obtained in \cite{bbgza}.

Coming back to \eqref{weight}, the natural starting point of the study is the investigation of the effect of $p$ on the first eigenvalue
$\lambda_1(p)$, namely to study:
$$
\inf_{p \in P_{\alpha, \beta}} \, \lambda_1(p).
$$
The minimization of the first eigenvalue for the second order Dirichlet version of \eqref{weight}, named \emph{composite membrane} problem, has been studied in \cite{chanillo1}-\cite{chanillo3},\cite{sha}, while the minimization of the first eigenvalue for the equation in \eqref{weight} under Dirichlet or Navier boundary conditions, named \emph{composite plate} problem, has been studied in \cite{anedda1},\cite{anedda2},\cite{CV}-\cite{cuccu2}. Finally, we refer to \cite{lapr} for a detailed stability analysis, upon variation of $p$, of the weighted eigenvalues of general elliptic operators of arbitrary order subject to several kinds of homogeneous boundary conditions. In this field of research, typical results are existence of optimal pairs and their qualitative properties, such as symmetry or symmetry breaking.
From this point of view a crucial obstruction, when passing from the membrane to the plate problem, namely from the second to the fourth order case,
is represented by the loss of maximum and comparison principles which usually enter either in the study of the simplicity of the first eigenvalue and
in the techniques applied to prove symmetry results, such as reflections methods or moving planes techniques. Nevertheless, a suitable choice of the
boundary conditions (e.g. Navier or Steklov b.c.) or of the geometry of the domain (e.g. small perturbations of balls) may yield the validity of
so-called \emph{positivity preserving property} which basically means that solutions, of the associated linear problem, maintain the sign of data.
Concerning problem \eqref{weight},
the difficulties in its analysis, are even increased by the choice of the unusual boundary conditions for which no positivity
preserving property is known. As far as we are aware, the minimization of the first eigenvalue of problem \eq{weight} has not been considered in literature, hence the present paper
represents the first contribution on this topic. In our analysis we take advantage of the fact that $\Omega$ is a planar domain and,
when restricting the class of weights, some explicit computations can be performed. On the other hand, we exploit a sort of restricted positivity
preserving property with respect to the $y$ variable, that we prove in Theorem \ref{corPPP} below, having its own theoretical interest. We note that
the above mentioned restriction on admissible weights is also justified by the applicative nature of our problem. Indeed, it is known that minimization problems,
like the composite membrane problem, naturally lead to homogenization \cite{murat}, see also \cite{ksw} for a stiffening problem for the torsion of a bar.
Homogenization would lead to optimal designs with reinforcements scattered throughout the structure, namely designs impossible to reproduce for engineers.
In order to avoid homogenization, the class of admissible reinforcements should be sufficiently small. See also Nazarov-Sweers-Slutskij \cite{nazarov},
where only ``macro'' reinforcements are considered, although in a fairly different setting.

As we have already remarked, in order to improve the stability of the plate, the final goal of the spectral analysis of problem \eq{weight} is to maximize the distance between selected oscillating modes. The present paper has to be meant as a first contribution in this direction and it should be hopefully followed by the optimization analysis of the higher eigenvalues. This, together with the further investigation of the positivity properties of the operator in \eq{weight}, represents a promising topic of research that we plan to develop in our future studies.

The paper is organised as follows. Section \ref{setting} is devoted to the description of the notations and of some results about the case
$p\equiv 1$. In Section \ref{s:main} one can find the main results of the paper which are proved in Sections \ref{s:proofs-1} and \ref{spectrum}. In Section \ref{numerics} we show some numerical results on the behaviour of the eigenvalues which complement our theoretical analysis. Finally, in Section \ref{pppproof}
we show the validity of a positivity preserving property for a one dimensional fourth order problem, coming from a suitable Fourier decomposition of solutions
to the plate problem.

\section{Notations and known results when $p\equiv 1$}\label{setting}
 From now onward we fix $\Omega=(0,\pi)\times(-\ell,\ell)\subset\R^2\,$ with $\ell>0$ and $0<\sigma<\frac{1}{2}$.
The natural functional space where to set problem \eq{weight} is
$$
H^2_*(\Omega)=\big\{u\in H^2(\Omega): u=0\mathrm{\ on\ }\{0,\pi\}\times(-\ell,\ell)\big\}\,.
$$
$H^2_*(\Omega)$ is a Hilbert space when endowed with the scalar product
$$
(u,v)_{H^2_*}:=\int_\Omega \left[\Delta u\Delta v+(1-\sigma)(2u_{xy}v_{xy}-u_{xx}v_{yy}-u_{yy}v_{xx})\right]\, dx \, dy \,
$$
and associated norm
$$
\|u\|_{H^2_*(\Omega)}^2=(u,u)_{H^2_*(\Omega)} \, ,
$$
which is equivalent to the usual norm in $H^2(\Omega)$, see \cite[Lemma 4.1]{fergaz}. 
 Then problem \eq{weight} may also be formulated in the following weak sense
\begin{equation}
\label{eigenweak}
(u,v)_{H^2_*(\Omega)} =\lambda\int_{\Omega}p(x,y)uv\,dx\,dy \qquad\forall v\in H^2_*(\Omega),
\end{equation}
where $p$ belongs to the family of weights $P_{\alpha,\beta}$ defined in \eqref{eq:famiglia} with $\alpha,\beta \in (0,+\infty)$ and $\alpha<\beta$ fixed. We point out that condition $p\in P_{\alpha,\beta}$ implies
$\alpha \leq 1 \leq \beta$ since $\int_\Omega p\,dx\,dy=|\Omega|$.
Moreover, we observe that it is not restrictive to assume
$\alpha<1<\beta$ when we focus our analysis on weights that do not
coincide a.e. with the constant function $p\equiv 1$. Indeed, if
we assume that $\beta=1$, it must be $p=1$ a.e. in $\Omega$ since
otherwise we would have $\int_\Omega p\,dx\,dy<|\Omega|$;
similarly, if we assume that $\alpha=1$, it must be $p=1$ a.e. in $\Omega$. For this reason, since the aim of our research is to study the
effect of a non constant weight on the first eigenvalue of
\eq{weight}, in what follows we will always assume
$\alpha<1<\beta$.

The bilinear form $(u,v)_{H^2_*}$ is continuous
and coercive and $p\in L^\infty(\Omega)$ is positive a.e. in $\Omega$, therefore
standard spectral theory of self-adjoint operators shows that the eigenvalues of \eq{weight} may be ordered in an increasing sequence of strictly positive
numbers diverging to $+\infty$ and that the corresponding eigenfunctions form a complete orthonormal system in $H^2_*(\Omega)$.
%

Since $p\in L^{\infty}(\Omega),$ by elliptic regularity the eigenfunctions are at least in $C^{2}(\overline \Omega).$ Furthermore, the first eigenvalue is characterized by
\begin{equation}\label{lambdaP}
 \lambda_1(p) :=  \, \min_{u \in H^2_*(\Omega) \setminus\{0\}} \frac{\|u\|_{H^2_*}^2}{\|\sqrt{p}\,u\|_{2}^2}.
\end{equation}

\par
When $p\equiv 1$ the spectrum of \eq{weight} has been completely characterized. We recall the following statement from \cite{fergaz},  including some refinements on the eigenvalues estimates proved in \cite{bebuga2}.

\begin{proposition}\label{eigenvalue}
Let $p\equiv 1$ in \eqref{weight}. The set of eigenvalues of \eqref{weight} may be ordered in an increasing sequence of strictly positive numbers diverging to $+\infty$
and any eigenfunction belongs to $C^\infty(\overline\Omega)$; the set of eigenfunctions of \eqref{weight} is a complete system in $H^2_*(\Omega)$. Moreover:\par\noindent
$(i)$ for any $m\ge1$, there exists a unique eigenvalue $\lambda=\mu_{m,1}\in((1-\sigma^2) m^4,m^4)$ with corresponding eigenfunction
$$\left[\big[\mu_{m,1}^{1/2}-(1-\sigma)m^2\big]\, \tfrac{\cosh\Big(y\sqrt{m^2+\mu_{m,1}^{1/2}}\Big)}{\cosh\Big(\ell\sqrt{m^2+\mu_{m,1}^{1/2}}\Big)}+
\big[\mu_{m,1}^{1/2}+(1-\sigma)m^2\big]\, \tfrac{\cosh\Big(y\sqrt{m^2-\mu_{m,1}^{1/2}}\Big)}{\cosh\Big(\ell\sqrt{m^2-\mu_{m,1}^{1/2}}\Big)}\right]\sin(mx)\, ;$$
$(ii)$ for any $m\ge1$ and any $k\ge2$ there exists a unique eigenvalue $\lambda=\mu_{m,k}>m^4$ satisfying

$\left(m^2+\frac{\pi^2}{\ell^2}\left(k-\frac{3}{2}\right)^2\right)^2<\mu_{m,k}<\left(m^2+\frac{\pi^2}{\ell^2}\left(k-1\right)^2\right)^2$

and with corresponding eigenfunction
$$
\left[\big[\mu_{m,k}^{1/2}-(1-\sigma)m^2\big]\, \tfrac{\cosh\Big(y\sqrt{\mu_{m,k}^{1/2}+m^2}\Big)}{\cosh\Big(\ell\sqrt{\mu_{m,k}^{1/2}+m^2}\Big)}
+\big[\mu_{m,k}^{1/2}+(1-\sigma)m^2\big]\, \tfrac{\cos\Big(y\sqrt{\mu_{m,k}^{1/2}-m^2}\Big)}{\cos\Big(\ell\sqrt{\mu_{m,k}^{1/2}-m^2}\Big)}\right]\sin(mx)\, ;
$$

$(iii)$ for any $m\ge1$ and any $k\ge2$ there exists a unique eigenvalue $\lambda=\nu_{m,k}>m^4$ with corresponding eigenfunctions
$$
\left[\big[\nu_{m,k}^{1/2}-(1-\sigma)m^2\big]\, \tfrac{\sinh\Big(y\sqrt{\nu_{m,k}^{1/2}+m^2}\Big)}{\sinh\Big(\ell\sqrt{\nu_{m,k}^{1/2}+m^2}\Big)}
+\big[\nu_{m,k}^{1/2}+(1-\sigma)m^2\big]\, \tfrac{\sin\Big(y\sqrt{\nu_{m,k}^{1/2}-m^2}\Big)}{\sin\Big(\ell\sqrt{\nu_{m,k}^{1/2}-m^2}\Big)}\right]\sin(mx)\, ;
$$
$(iv)$ for any $m\ge1$ satisfying $\ell m\sqrt 2\, \coth(\ell m\sqrt2 )>\left(\frac{2-\sigma}{\sigma}\right)^2$ there exists a unique
eigenvalue $\lambda=\nu_{m,1}\in(\mu_{m,1},m^4)$ with corresponding eigenfunction
$$\left[\big[\nu_{m,1}^{1/2}-(1-\sigma)m^2\big]\, \tfrac{\sinh\Big(y\sqrt{m^2+\nu_{m,1}^{1/2}}\Big)}{\sinh\Big(\ell\sqrt{m^2+\nu_{m,1}^{1/2}}\Big)}
+\big[\nu_{m,1}^{1/2}+(1-\sigma)m^2\big]\, \tfrac{\sinh\Big(y\sqrt{m^2-\nu_{m,1}^{1/2}}\Big)}{\sinh\Big(\ell\sqrt{m^2-\nu_{m,1}^{1/2}}\Big)}\right]
\sin(mx)\, .$$
Finally, if
\begin{equation}\label{c0}
\text{the unique positive solution $s>0$ of: }\tanh(\sqrt{2}s\ell)=\left(\frac{\sigma}{2-\sigma}\right)^2\, \sqrt{2}s\ell\quad  \text{is not an integer,}
\end{equation}
then the only eigenvalues are the ones given in $(i)-(iv)$.

\end{proposition}

In the following, to avoid too many distinctions, we will always assume that \eq{c0} holds.

By Proposition \ref{eigenvalue} and \cite[Section 7]{fergaz} it is readily deduced that the first eigenvalue of problem \eqref{weight} with $p\equiv1$ is $\mu_{1,1}$,
namely $ \lambda_1(1)=\mu_{1,1}$, it is simple and up to constant multiplier the first eigenfunction is given by
\begin{equation}\label{u1}
u_1(x,y)= \left[\big[\mu_{1,1}^{1/2}-(1-\sigma)\big]\, \tfrac{\cosh\Big(y\sqrt{1+\mu_{1,1}^{1/2}}\Big)}{\cosh\Big(\ell\sqrt{1+\mu_{1,1}^{1/2}}\Big)}+
\big[\mu_{1,1}^{1/2}+(1-\sigma)\big]\, \tfrac{\cosh\Big(y\sqrt{1-\mu_{1,1}^{1/2}}\Big)}{\cosh\Big(\ell\sqrt{1-\mu_{1,1}^{1/2}}\Big)}\right] \sin x\,.
\end{equation}
Hence, $u_1$ is positive in $\Omega$, convex in the $y-$variable and concave in the $x-$variable.  \par

\section{Main results} \label{s:main}
As in Section \ref{setting} we always assume
$$
0<\sigma<\frac{1}{2}\quad\text{and}\quad \alpha<1<\beta \quad (\alpha,\beta\in (0,+\infty)).
$$
Then, recalling \eqref{eq:famiglia}, we focus on the infimum problem
\begin{equation}\label{CP}
\lambda_{\alpha,\beta} :=\inf_{p \in P_{\alpha,\beta}} \, \lambda_1(p),
\end{equation}
where $\lambda_1(p)$ is defined in \eqref{lambdaP}.
\begin{definition}
A couple $ (\overline{p},u_{\overline p}) \in P_{\alpha,\beta} \times H^2_*(\Omega)$ is called {\em optimal pair} if $\overline p$ achieves the infimum in \eqref{CP} and $u_{\overline p}$ is an eigenfunction of $\lambda_1(\overline p)$ .
\end{definition}

Adapting to our case \cite[Theorem 13]{chanillo} and \cite[Theorem 1.4]{CV}, we show that there exists an {\rm optimal pair}  $ (\overline{p},u_{\overline p}) $ for problem \eqref{CP} and $u_{\overline p}$ and $\overline{p}$ are suitably related. Using the language of the control theory we find that $\overline p $ is a \textit{bang-bang} function; more precisely we prove

\begin{theorem}\label{thm-exist-qual}
There exists and optimal pair $(\overline{p},u_{\overline p}) \in P_{\alpha,\beta} \times H^2_*(\Omega)$.
Furthermore, $\overline{p}$ and $u_{\overline p}$ are related as follows
\begin{equation}\label{pS}
\overline{p}(x,y) = \alpha \chi_{S} (x,y)+ \beta \chi_{\Omega \setminus S}(x,y)\,\quad \text{for a.e. }\quad (x,y)\in \Omega\,,
\end{equation}
where $\chi_{S}$ and $ \chi_{\Omega \setminus S}$ are the characteristic functions of the sets $S$ and $\Omega \setminus S$ and $S\subset \Omega$
is such that $|S|=\frac{\beta-1}{\beta-\alpha}\,|\Omega|$ and $S = \{ (x,y)\in \Omega\,:\,u_{\overline p}^2(x,y) \leq t \}$ for some $t> 0$.
\end{theorem}

Theorem \ref{thm-exist-qual} states that the plate has to be made of two materials, but it gives no information about the location of the materials and hence, no practical information on how to build the plate. To this aim, a more explicit suggestion is provided by the following

\begin{proposition}\label{comparison_esy}
Let $p\in P_{\alpha,\beta}$ satisfy one of the following assumptions
\begin{itemize}
\item[(i)] $p=p(y)$ is even and there exists $z\in (0, \ell)$ such that
$$p(y)\leq 1 \quad \text{for } y\in [0,z] \quad \text{ and } \quad p(y)\geq 1 \quad \text{for } y\in [z,\ell)\,.$$
\item[(ii)] $p=p(x)$ is symmetric with respect to the line $x=\frac{\pi}{2}$ and there exists $s \in (0, \frac{\pi}{2})$ such that
$$p(x)\leq 1 \quad \text{for } x\in (0,s] \quad \text{ and } \quad p(x)\geq 1 \quad \text{for } x\in [s,\frac{\pi}{2}]\,.$$
\end{itemize}
Then,
\begin{equation}\label{comparison}
\lambda_1(p)\leq \lambda_1(1)=\mu_{1,1}\,,
\end{equation}
where $\mu_{1,1}$ is as defined in Proposition \ref{eigenvalue}-(i).
\end{proposition}

\begin{remark}
The same idea of the proof of Proposition \ref{comparison_esy}-(i) can be repeated to prove that \eqref{comparison} holds if
$p\in P_{\alpha,\beta}$ satisfies
\begin{itemize}
\item[(iii)] $p=p(y)$ is even and there exist $2N+2$ points $0=y_0<y_1<y_2<...<y_{2N+2}=\ell$ such that
$$p(y)\leq 1 \quad \text{for } y\in [y_{2h}, y_{2h+1}]\,, \quad p(y)\geq 1 \quad \text{for } y\in  [y_{2h+1}, y_{2h+2}] \quad \text{ and } \int_{y_{2h}}^{y_{2h+2}}(p-1)\, dy=0 \,,$$
for all $h=0,...,N$.
\end{itemize}
\end{remark}

Since the weights considered in Proposition \ref{comparison_esy}  prove to be effective in lowering the first frequency of \eq{weight0},  by combining Proposition \ref{comparison_esy} with Theorem \ref{thm-exist-qual}, we include in the list of candidate solutions to problem \eqref{CP} the weights:
\begin{equation}\label{opt}
\overline{p}(y) = \alpha \chi_{(-\frac{\ell(\beta-1)}{\beta-\alpha},\frac{\ell(\beta-1)}{\beta-\alpha})}(y)
+ \beta \chi_{ (-\ell, \ell)\setminus (-\frac{\ell(\beta-1)}{\beta-\alpha},\frac{\ell(\beta-1)}{\beta-\alpha})}(y) \qquad y\in(-\ell, \ell)
\end{equation}
and
\begin{equation*}\label{opt2}
\overline{p}(x)=\beta \chi_{(\frac{\pi}{2} \frac{\beta-1}{\beta-\alpha},\frac{\pi}{2} \frac{\beta-2\alpha+1}{\beta-\alpha})}(x)
+\alpha \chi_{ (0,\pi)\setminus(\frac{\pi}{2} \frac{\beta-1}{\beta-\alpha},\frac{\pi}{2} \frac{\beta-2\alpha+1}{\beta-\alpha})}(x)\qquad x\in(0, \pi) \, .
\end{equation*}
\begin{figure}[h!tb]
	{\includegraphics[scale=0.2]{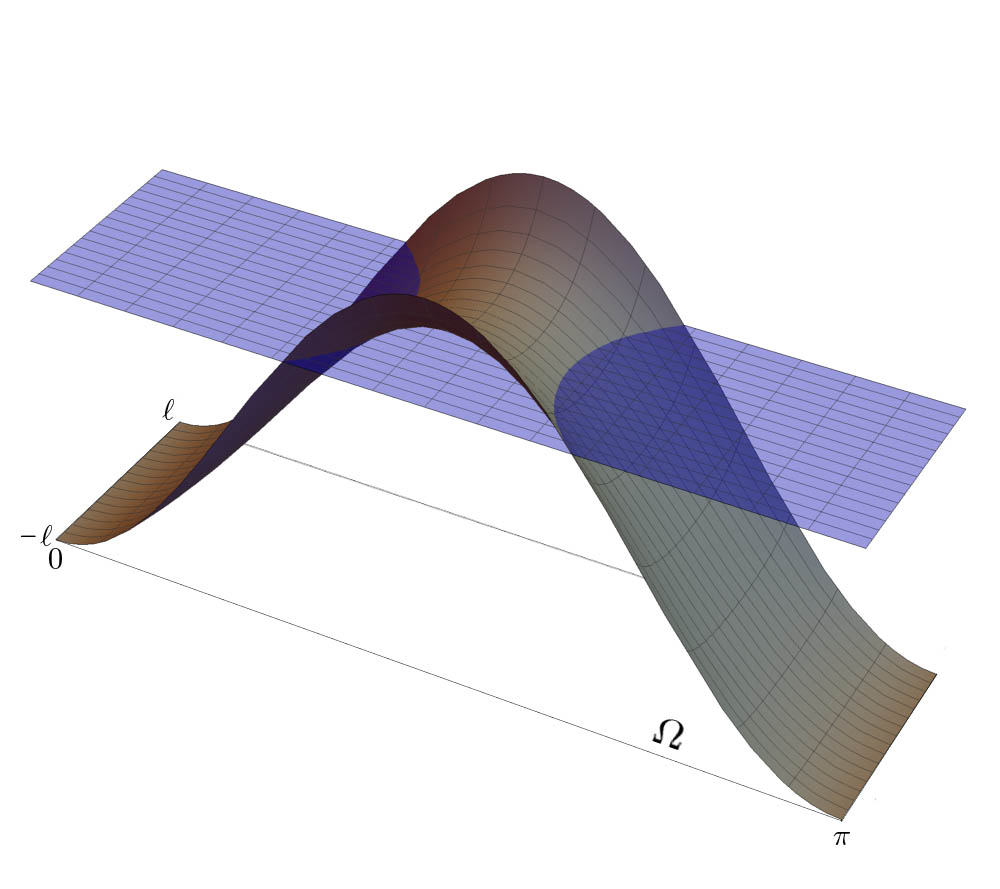}}\qquad	{\includegraphics[scale=0.7]{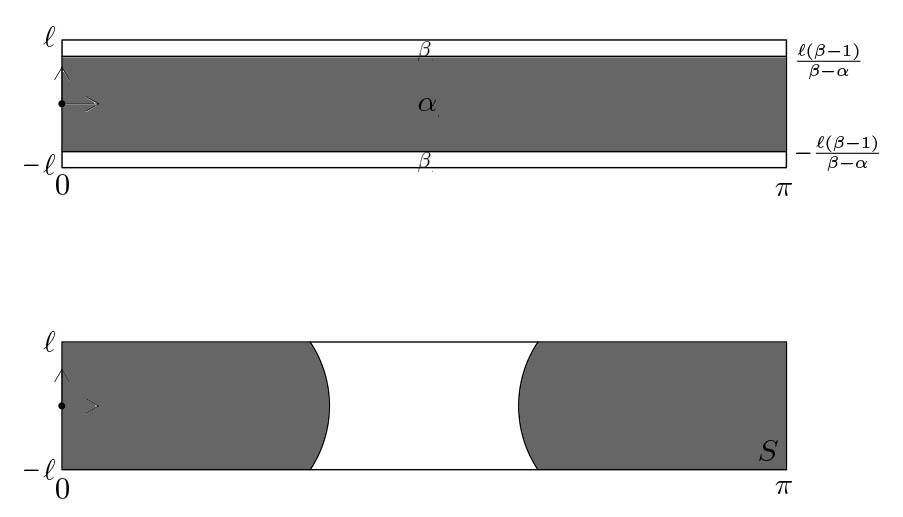}}
	\caption{On the left, plot of the eigenfunction $u_{1,\overline p}^2(x,y)$, corresponding to $\lambda_1(\overline p)$ with $\overline{p}(y)$ as in \eqref{opt}, intersected with $t>0$.
On the right, plot of $\overline{p}(y)$ (top) and plot of the sublevel set $S = \{ (x,y)\in \Omega\,:\,u_{1,\overline p}^2(x,y) \leq t \}$ (bottom).}
	\label{planar}
\end{figure}

 In Section \ref{numerics} we obtain numerically a positive eigenfunction, denoted by $u_{1,\overline p}(x,y)$, corresponding to $\lambda_1(\overline p)$ with $\overline{p}(y)$
as in \eqref{opt}. In Figure \ref{planar} on the left, we plot $z=u_{1,\overline p}^2(x,y)$ and we use it to determine qualitatively what should be the set $S$
predicted by Theorem \ref{thm-exist-qual}. In Figure \ref{planar} on the right we compare the weight $\overline p(x,y)$ in \eqref{pS} (bottom), with this choice of the set $S$, and the weight $\overline{p}(y)$ (top). From these plots we infer that $(\overline{p}(y),u_{1,\overline p}(x,y))$ is not a theoretical optimal pair of \eqref{CP}.

On the other hand, in Theorem \ref{min} below we prove that $\overline{p}(y)$ belongs to an optimal pair if we properly restrict the class of admissible weights to a suitable subset of $P_{\alpha,\beta}$.
\par

\begin{theorem}\label{min}
	Let us define
\begin{align*}
& \overline P_{\alpha,\beta} =\{p\in P_{\alpha,\beta}\,: \, p=p(y) \text{ is even, } p \text{ is piecewise continuous in } (-\ell,\ell)  \\[8pt]
& \qquad \text{ and }
\exists \, z\in (0, \ell)\,: p(y)\leq 1  \text{ in } [0,z]\,, \, p(y)\geq 1 \text{ in } [z,\ell)\} \, .
\end{align*}
When
$\beta<\min\{1/\mu_{1,1}\,,16(1-\sigma^2) \}$
the following statements hold:
\begin{itemize}
 \item[$(i)$] if $p_1, p_2\in \overline P_{\alpha,\beta}$ and there exists $z \in (0, \ell)$ such that
$$p_1(y)\leq p_2(y) \quad \text{in } [0,z]\quad \text{ and }\quad p_1(y)\geq p_2(y) \quad \text{in } [z,\ell)\,,$$
then
$$\lambda_1(p_1)\leq \lambda_1(p_2)\,;$$
  \item[$(ii)$] we have

$$\min_{p\in \overline P_{\alpha,\beta}} \lambda_1(p)=\lambda_1(\overline p)\,,$$
where $\overline p$ is as defined in \eqref{opt}.
\end{itemize}
\end{theorem}

\begin{remark}\label{oderem} Concerning the meaning of the upper bound $\beta<\min\{1/\mu_{1,1}\,,16(1-\sigma^2) \}$ in Theorem \ref{min} a couple of remarks are in order. The proof of Theorem \ref{min} is achieved by studying a family of related 1-dimensional eigenvalue problems. Indeed, any eigenfunction of \eqref{weight} can be expanded in Fourier series as follows
$$
u(x,y)=\sum_{m=1}^{+\infty} \varphi_m(y)\sin(mx)
$$
with $\varphi_m \in C^2([-\ell,\ell])$ and, if $p=p(y)$, for every $m\geq 1$ fixed, $\varphi_m$ satisfies the weak form of the problem:
\begin{equation}\label{weight1d}
\begin{cases}
\varphi''''(y)-2m^2\varphi''(y)+m^4\varphi(y)=\lambda p(y)\varphi(y) & \qquad \text{in } (-\ell,\ell) \\
\varphi''(\pm\ell)-\sigma m^2\varphi(\pm\ell)=0 & \qquad  \\
\varphi'''(\pm\ell)-(2-\sigma)m^2\varphi'(\pm\ell)=0\,.& \qquad \,
\end{cases}
\end{equation}
See Section \ref{spectrum} for the details. In particular, if we denote by $\lambda_1(p)$ the first eigenvalue of \eqref{weight} and by $\overline\lambda_1(p,m)$ the first eigenvalue of \eq{weight1d} with $m\geq 1$ fixed, assumption $\beta\leq 16(1-\sigma^2)$ ensures that
\begin{equation}\label{monot1}
\lambda_1(p)=\min_{m\geq 1}\left\{\overline\lambda_1(p,m)\right\}=\overline\lambda_1(p,1) \, ,
\end{equation}
see Lemma \ref{comp1} below. On the other hand, the condition $\beta<1/\mu_{1,1}$ allows to prove that the first eigenfunction of $\overline\lambda_1(p,1)$ is monotone in $(0,\ell)$, and this information yields the comparison between weights of Theorem \ref{min}-$(i)$, see Lemma \ref{convex}.  The numerical results we state in Section \ref{numerics} suggest that both the upper bounds on $\beta$ are merely technical conditions.
\end{remark}

 It is worth noting that, in order to lower the first eigenvalue of $\Delta^2$ under Dirichlet or Navier boundary conditions, since the eigenfunctions vanish on the boundary,
 one expects that  the weight is more effective if it achieves its lowest value close to the boundary, see e.g. \cite[Theorem 1.5]{CV}. Theorem \ref{min} shows that the partially hinged
 boundary conditions lead to a complete different situation since the weight $\overline p(y)$ achieves its lowest value $\alpha$ far from the free long edges, see Figure \ref{planar} on the right (top).
 This behaviour is somehow related to the monotonicity of the first eigenfunction, as shown by Theorem \ref{signTH} below, cfr. Figure \ref{eig1}.

\begin{theorem}\label{signTH}
Let $\overline P_{\alpha,\beta}$ be the family of weights defined in Theorem \ref{min} with $\beta<\min\{1/\mu_{1,1}\,,16(1-\sigma^2) \}$.
Then, for any $p\in \overline P_{\alpha,\beta}$ the first eigenvalue $\lambda_1(p)$ of \eq{eigenweak} is simple. Furthermore, if $u_{1,p}$ is an eigenfunction of $\lambda_1(p)$ then
$u_{1,p}$ is of one
sign in $\Omega$ and moreover $u_{1,p}$ can be written as $u_{1,p}(x,y)=\varphi_{1,p}(y)\sin (x)$ with $\varphi_{1,p}(y)$ even and strictly monotone in $(0,\ell)$.
\end{theorem}
\begin{figure}[h!tb]
{\includegraphics[scale=0.3]{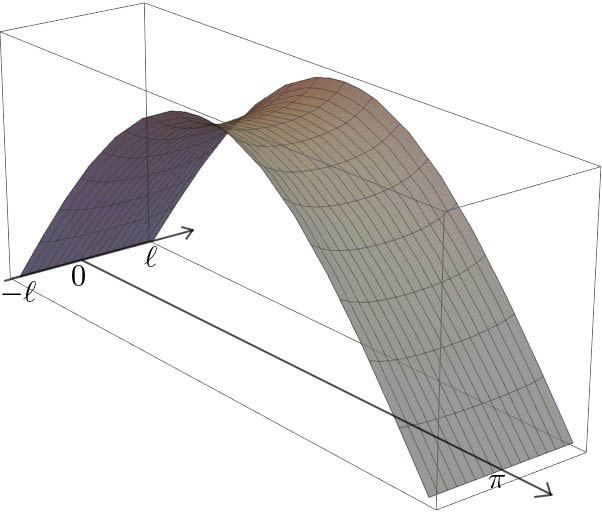}}
\caption{Qualitative plot of $u_{1,p}(x,y)=\varphi_{1,p}(y)\sin (x)$.}
\label{eig1}
\end{figure}

\par
Unfortunately, the above statement does not carry over to all weights $p\in P_{\alpha,\beta}$. This is related to the well-know loss of  comparison principles
for fourth order elliptic operators. Indeed, the proof of Theorem \ref{signTH} strongly relies on a sort of restricted positivity preserving property with respect
to the $y$ variable that we prove by separating variables. More precisely, we have

\begin{theorem}\label{corPPP}
 Let $m\geq 1$ be an integer. Furthermore, let $u\in H^2_*(\Omega)$ be the weak solution to the problem
 \begin{equation*}\label{linear}
\begin{cases}
\Delta^2 u= f(y)\, \sin(mx) & \qquad \text{in } \Omega\, \\
u(0,y)=u_{xx}(0,y)=u(\pi,y)=u_{xx}(\pi,y)=0 & \qquad \text{for } y\in (-\ell,\ell)\,\\
u_{yy}(x,\pm\ell)+\sigma
u_{xx}(x,\pm\ell)=u_{yyy}(x,\pm\ell)+(2-\sigma)u_{xxy}(x,\pm\ell)=0
& \qquad \text{for } x\in (0,\pi)\, ,
\end{cases}
\end{equation*}
namely
$$
(u,v)_{H^2_*} =\int_{\Omega} f(y)\sin(mx)\, v \qquad\forall v\in H^2_*(\Omega)\,.
$$
Then, $u(x,y)=w_m(y)\sin(mx)$ for some $w_m\in H^2(-\ell, \ell)$ and the following implication holds
 $$f\geq 0\text{ in } (-\ell, \ell)\,(f\not \equiv 0)\quad \Rightarrow \quad w_m(y)>0 \text{ in } [-\ell, \ell]\,.$$
\end{theorem}

\section{Numerical Results} \label{numerics}
In this section, for any $m\geq 1$, we compute numerically the first eigenvalue $\overline{\lambda}_1(p,m)$ of problem \eqref{weight1d} with $p$ as defined in \eqref {opt}. More precisely, we take
\begin{equation*}\label{load}
\overline p_{\alpha, \beta}(y)=
\begin{cases}
\beta\qquad y\in (-\ell,-\overline{y})\cup (\overline{y},\ell)\\
\alpha \qquad y\in (-\overline{y},\overline{y})
\end{cases}
\end{equation*}
with $0<\alpha<1<\beta$ and $\overline{y}=\frac{\ell(\beta-1)}{\beta-\alpha}$, so that $\int_{0}^{\ell}pdy=\ell$. In terms of engineering applications, this means that we are dealing with a weight given by the pairing of two materials having different densities $\alpha$
and $\beta$, properly placed on rectangular strips, having the length of the whole plate. Furthermore, we assume that the deck of the bridge is composed by a mixture of concrete and steel, hence the Poisson ratio is variable between 0.15 and 0.3; for this reason in what follows we take $\sigma=0.2$.
 \par
Note that, since $\overline p_{\alpha, \beta}(y)$ is an even function, to determine all eigenvalues of \eqref{weight1d}, we may focus on even and odd eigenfunctions.
Indeed, if $\varphi(y)$ is an eigenfunction which is neither odd or even, it is readily verified that also $\varphi^{ev}(y):=\frac{\varphi(y)+\varphi(-y)}{2}$
and $\varphi^{od}(y):=\frac{\varphi(y)-\varphi(-y)}{2}$ are eigenfunctions, respectively even and odd, corresponding to the same eigenvalue of $\varphi(y)$.
On the other hand, since the first eigenvalue of \eqref{weight1d} is simple and the corresponding eigenfunctions are of one sign in $[-\ell,\ell]$, see Remark \ref{oderem} and Theorem \ref{signTH},
we infer that $\varphi$ must be an even function, whence to compute $\overline{\lambda}_1(p,m)$ we may concentrate on even eigenfunctions that we named $\varphi^{ev}$.
For any $m\geq 1$ we have that
\begin{equation}\label{phieven}
\varphi^{ev}(y)=\begin{cases}
h_1(-y)\qquad&{\rm on}\hspace{2mm}[-\ell,-\overline{y}]\\
h_2(y)\qquad&{\rm on}\hspace{2mm}(-\overline{y},\overline{y})\\
h_1(y)\qquad&{\rm on}\hspace{2mm}[\overline{y},\ell]\\
\end{cases}\end{equation}
where $h_1$ and $h_2$ satisfy:
\begin{equation}\label{num1}
\begin{cases}
h_1''''(y)-2m^2h_1''(y)+m^4h_1(y)=\lambda \beta h_1(y)\qquad {\rm on}\hspace{2mm}(\overline{y},\ell)\\
h_2''''(y)-2m^2h_2''(y)+m^4h_2(y)=\lambda \alpha h_2(y)\qquad {\rm on}\hspace{2mm}[0,\overline{y})\vspace{2mm}\\
h_1''(\ell)-\sigma m^2h_1(\ell)=0, \qquad
h_1'''(\ell)-(2-\sigma)m^2h_1'(\ell)=0,\\
h_2'(0)=0, \hspace{28mm}
h_2'''(0)=0,
\\ h_1(\overline{y})=h_2(\overline{y}),\hspace{21mm} h'_1(\overline{y})=h'_2(\overline{y}),\\h''_1(\overline{y})=h''_2(\overline{y}),\hspace{21mm} h'''_1(\overline{y})=h'''_2(\overline{y})\,.
\end{cases}
\end{equation}
Note that the compatibility conditions between the functions $h_1$ and $h_2$, ensure that $\varphi^{ev} \in C^3([-\ell,\ell])$, while $h_2'(0)=h_2'''(0)=0$ come from $\varphi^{ev}(-y)=\varphi^{ev}(y)$ and its regularity.
Clearly, the analytical expression of $h_1(y)$ and $h_2(y)$ depends on the roots of the characteristic polynomials related to the first two equations in \eqref{num1};
we denote them  respectively by $\zeta_1$ and $\zeta_2$ and we find that they satisfy
\begin{equation*}
\zeta_1^2=m^2\pm\sqrt{\lambda\beta}\qquad\zeta_2^2=m^2\pm\sqrt{\lambda\alpha}.
\end{equation*}
Therefore, the sign of $m^2-\sqrt{\lambda\beta}$ and $m^2-\sqrt{\lambda \alpha}$ determines different kinds of solutions. We introduce the following notations
$$
\eta_\alpha:=\sqrt{m^2+\sqrt{\lambda \alpha}},\qquad\eta_\beta:=\sqrt{m^2+\sqrt{\lambda \beta}},\qquad\omega_\alpha
:=\sqrt{|m^2-\sqrt{\lambda \alpha}|},\qquad\omega_\beta:=\sqrt{|m^2-\sqrt{\lambda \beta}|} \, ,
$$
and we distinguish five cases:
\begin{itemize}
	\item[\textbf{a)}] $m^4>\lambda\beta>\lambda\alpha$, implying $\lambda<m^4/\beta$ and \\
	\begin{equation*}
	\begin{split}
	&h_1(y)=a_1\cosh\big(\eta_\beta y\big)+b_1\sinh\big(\eta_\beta y\big)+c_1\cosh\big(\omega_\beta y\big)+d_1\sinh\big(\omega_\beta y\big)\,,\\
	&h_2(y)=a_2\cosh\big(\eta_\alpha y\big)+c_2\cosh\big(\omega_\alpha y\big)\,,
	\end{split}
	\end{equation*}
	\item[\textbf{b)}] $m^4=\lambda\beta$, so that $\eta_\alpha=m\sqrt{1+\sqrt{\alpha/\beta}}$, $\omega_\alpha=m\sqrt{1-\sqrt{\alpha/\beta}}$ and\\
	\begin{equation*}
	\begin{split}
	&h_1(y)=a_1\cosh\big(\sqrt{2}m y\big)+b_1\sinh\big(\sqrt{2}m y\big)+c_1y+d_1\,,\\
	&h_2(y)=a_2\cosh\big(\eta_\alpha y\big)+c_2\cosh\big(\omega_\alpha y\big)\,,
	\end{split}
	\end{equation*}
	\item[\textbf{c)}] $\lambda\alpha<m^4<\lambda\beta$, implying $m^4/\beta<\lambda<m^4/\alpha$ and \\
	\begin{equation*}
	\begin{split}
	&h_1(y)=a_1\cosh\big(\eta_\beta y\big)+b_1\sinh\big(\eta_\beta y\big)+c_1\cos\big(\omega_\beta y\big)+d_1\sin\big(\omega_\beta y\big)\,,\\
	&h_2(y)=a_2\cosh\big(\eta_\alpha y\big)+c_2\cosh\big(\omega_\alpha y\big)\,,\end{split}
	\end{equation*}
	\item[\textbf{d)}] $m^4=\lambda\alpha$, so that $\eta_\beta=m\sqrt{1+\sqrt{\beta/\alpha}}$, $\omega_\beta=m\sqrt{\sqrt{\beta/\alpha}-1}$ and \\
	\begin{equation*}
	\begin{split}
	&h_1(y)=a_1\cosh\big(\eta_\beta y\big)+b_1\sinh\big(\eta_\beta y\big)+c_1\cos\big(\omega_\beta y\big)+d_1\sin\big(\omega_\beta y\big)\,,\\
	&h_2(y)=a_2\cosh\big(\sqrt{2}m y\big)+c_2\,,
	\end{split}
	\end{equation*}
	\item[\textbf{e)}] $m^4<\lambda\alpha<\lambda\beta$, implying $\lambda>m^4/\alpha$ and \\
	\begin{equation*}
	\begin{split}
	&h_1(y)=a_1\cosh\big(\eta_\beta y\big)+b_1\sinh\big(\eta_\beta y\big)+c_1\cos\big(\omega_\beta y\big)+d_1\sin\big(\omega_\beta y\big)\,,\\
	&h_2(y)=a_2\cosh\big(\eta_\alpha y\big)+c_2\cos\big(\omega_\alpha y\big)\,.
	\end{split}
	\end{equation*}
\end{itemize}
The six coefficients involved in the definition of $h_1$ and $h_2$ can be determined, in each of the five cases, by imposing the boundary and compatibility conditions. We present here only case $\textbf{c)}$, since the others cases can be treated similarly.\par
First of all we assume that $h_1$ satisfies the boundary conditions, i.e.
\begin{equation*}
(BCs)
\begin{cases}
h_1''(\ell)-\sigma m^2h_1(\ell)=0\\
h_1'''(\ell)-(2-\sigma)m^2h_1'(\ell)=0
\end{cases}
\qquad\Rightarrow\qquad
\begin{cases}
(\eta_\beta^2-\sigma m^2)[a_1\cosh(\eta_\beta\ell)+b_1\sinh(\eta_\beta\ell)]+\\-(\omega_\beta^2+\sigma m^2)[c_1\cos(\omega_\beta\ell)+d_1\sin(\eta_\beta\ell)]=0\\
(\eta_\beta^2+(\sigma-2) m^2)\eta_\beta[a_1\sinh(\eta_\beta\ell)+b_1\cosh(\eta_\beta\ell)]+\\
(\omega_\beta^2-(\sigma-2) m^2)\omega_\beta[c_1\sin(\omega_\beta\ell)-d_1\cos(\omega_\beta\ell)]=0,
\end{cases}
\end{equation*}
then we impose the compatibility conditions, i.e.
\small
\begin{equation*}
\begin{split}
i)\\ii)\\iii)\\iv)
\end{split}
\begin{cases}
h_1(\overline{y})=h_2(\overline{y})\\ h'_1(\overline{y})=h'_2(\overline{y})\\h''_1(\overline{y})=h''_2(\overline{y})\\ h'''_1(\overline{y})=h'''_2(\overline{y})
\end{cases}
\qquad\Rightarrow\qquad
\begin{cases}
a_1\cosh\big(\eta_\beta \overline{y}\big)+b_1\sinh\big(\eta_\beta \overline{y}\big)+c_1\cos\big(\omega_\beta \overline{y}\big)
+d_1\sin\big(\omega_\beta \overline{y}\big)+\\-a_2\cosh\big(\eta_\alpha \overline{y}\big)-c_2\cosh\big(\omega_\alpha \overline{y}\big)=0
\\
a_1\eta_\beta\sinh\big(\eta_\beta \overline{y}\big)+b_1\eta_\beta\cosh\big(\eta_\beta \overline{y}\big)
-c_1\omega_\beta\sin\big(\omega_\beta \overline{y}\big)+d_1\omega_\beta\cos\big(\omega_\beta \overline{y}\big)+\\
-a_2\eta_\alpha\sinh\big(\eta_\alpha \overline{y}\big)-c_2\omega_\alpha\sinh\big(\omega_\alpha \overline{y}\big)=0
\\
a_1\eta_\beta^2\cosh\big(\eta_\beta \overline{y}\big)+b_1\eta_\beta^2\sinh\big(\eta_\beta \overline{y}\big)
-c_1\omega_\beta^2\cos\big(\omega_\beta \overline{y}\big)-d_1\omega_\beta^2\sin\big(\omega_\beta \overline{y}\big)+\\
-a_2\eta_\alpha^2\cosh\big(\eta_\alpha \overline{y}\big)-c_2\omega_\alpha^2\cosh\big(\omega_\alpha \overline{y}\big)=0\\
a_1\eta_\beta^3\sinh\big(\eta_\beta \overline{y}\big)+b_1\eta_\beta^3\cosh\big(\eta_\beta \overline{y}\big)
+c_1\omega_\beta^3\sin\big(\omega_\beta \overline{y}\big)-d_1\omega_\beta^3\cos\big(\omega_\beta \overline{y}\big)+
\\-a_2\eta_\alpha^3\sinh\big(\eta_\alpha \overline{y}\big)-c_2\omega_\alpha^3\sinh\big(\omega_\alpha \overline{y}\big)=0.
\end{cases}
\end{equation*}
\normalsize
We should solve a system of six equations and six unknowns; through some algebraic manipulations, we reduce it to a system of four equations and four unknowns $\textbf{v}=(a_1,b_1,c_1,d_1)^T$. More precisely, we get
{\small \begin{equation}\label{num2}
	\begin{cases}
	(BCs)\\
	[\eta_\alpha^2 (h_1(\overline{y})-h_2(\overline{y}))-(h''_1(\overline{y})-h''_2(\overline{y}))]\omega_\alpha\sinh(\omega_\alpha\overline{y})=[\eta_\alpha^2(h'_1(\overline{y})-h'_2(\overline{y}))-(h'''_1(\overline{y})-h'''_2(\overline{y}))]\cosh(\omega_\alpha\overline{y})\\
	[\omega_\alpha^2 (h_1(\overline{y})-h_2(\overline{y}))-(h''_1(\overline{y})-h''_2(\overline{y}))]\eta_\alpha\sinh(\eta_\alpha\overline{y})=[\omega_\alpha^2(h'_1(\overline{y})-h'_2(\overline{y}))-(h'''_1(\overline{y})-h'''_2(\overline{y}))]\cosh(\eta_\alpha\overline{y}).
	\end{cases}
	\end{equation}}
To system \eqref{num2} we associate a square matrix depending on the eigenvalues  $\textbf{M}(\lambda)\in \mathbb{M}_4(\mathbb{R})$, hence \eqref{num2} rewrites $\textbf{M}(\lambda) \textbf{v}=\textbf{0}$; since we are interested in not trivial solutions we end up with the equation
\begin{equation}\label{num3}
f(\lambda):=\det\textbf{M}(\lambda)=0 \quad \text{with } \lambda >0.
\end{equation}
In this way, for any $m\geq 1$ fixed, the zeroes of the function $f(\lambda)$ in the interval $m^4/\beta<\lambda<m^4/\alpha$, if they exist, are the eigenvalues corresponding to eigenfunctions $\varphi^{ev}$ as in \eqref{phieven} with $h_1$ and $h_2$ as in $\textbf{c)}$. This procedure can be applied to each of the five cases $\textbf{a)}-\textbf{e)}$. \par

The computation by hand of \eqref{num3} is very involved, thus we perform it numerically in all the five cases listed above. Our experiments reveal that cases $\textbf{b)}$ and $\textbf{d)}$ do not occur if $1\leq m\leq M$, for a suitable $M$ which, for all tested values of $\alpha$ and $\beta$, satisfies $M\approx6/\ell$. This implies large $M$ for small $\ell$, as common in plates for bridges. Therefore, we focus on cases \textbf{a)}-\textbf{c)}-\textbf{e)}.
\begin{figure}[h!tb]
	\centerline{\includegraphics[scale=0.3]{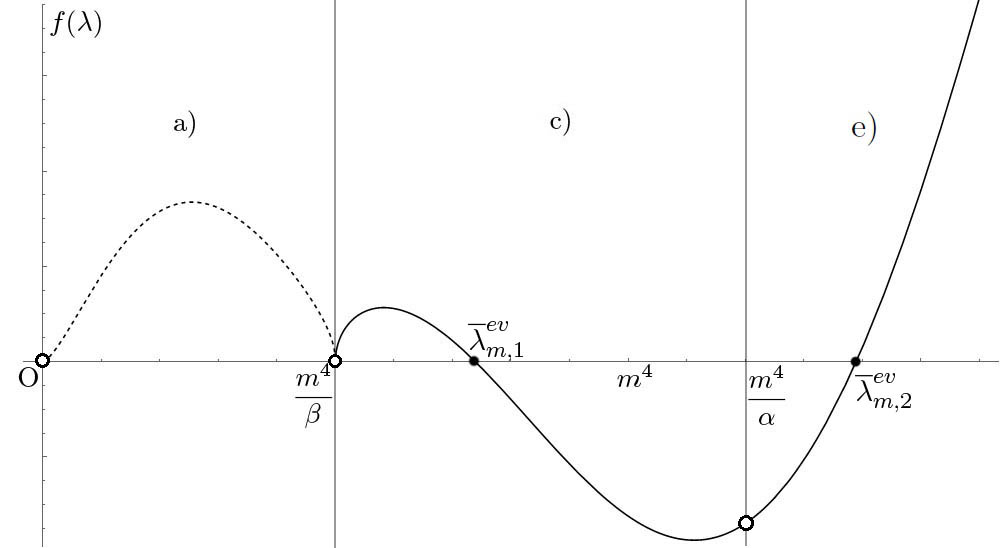}}
	\caption{Plot of $f(\lambda)$ in the cases a) (dashed), c) and e). Here $\overline{\lambda}^{ev}_{m,k}:=\overline{\lambda}_k^{ev}(\overline p_{\alpha, \beta},m)$.}
	\label{eigplot}
\end{figure}
We tested several values of $0<\alpha<1<\beta$ and $1\leq m\leq M$ always obtaining the same qualitative plot of $f(\lambda)$.  Figure \ref{eigplot} shows the following facts: we do not find eigenvalues in case $\textbf{a)}$, since $f(\lambda)>0$ for all $\lambda\in(0,m^4/\beta)$; the first eigenvalue $\overline{\lambda}^{ev}_1(\overline p_{\alpha, \beta},m)$ falls always in case $\textbf{c)}$; all the other eigenvalues corresponding to even functions fall in case $\textbf{e)}$. Furthermore, our numerical results yield the following bounds on eigenvalues corresponding to even eigenfunctions:
$$\dfrac{m^4}{\beta}<\overline{\lambda}^{ev}_1(\overline p_{\alpha, \beta},m)=\overline{\lambda}_1(\overline p_{\alpha, \beta},m)<m^4,\qquad \overline{\lambda}^{ev}_k(\overline p_{\alpha, \beta},m)>\dfrac{m^4}{\alpha}\qquad{\rm for}\hspace{2mm}k\geq2\,.$$

\par
\bigskip \par
We are now interested in checking if \eqref{monot1} holds when the upper bound on $\beta$ of Theorem \ref{min} is not satisfied (see also Lemma \ref{comp1} below), i.e. if
$$\overline{\lambda}_1^{ev}(\overline p_{\alpha, \beta},m) \geq \overline{\lambda}_1^{ev}(\overline p_{\alpha, \beta},1) \quad \text{for } m\geq 2$$
when $\beta\gg 16(1-\sigma^2)$. To this aim we study the behaviour of the maps $\beta\mapsto \overline{\lambda}^{ev}_1(\overline p_{\alpha, \beta},m)$ and $m \mapsto \overline{\lambda}^{ev}_1(\overline p_{\alpha, \beta},m)$. In Figure \ref{eigplot1} we plot some points of the map $\beta\mapsto\overline{\lambda}_1^{ev}(\overline p_{\alpha, \beta},1) $ for $\alpha=0.5$, we register a similar behaviour for $\overline{\lambda}_1^{ev}(\overline p_{\alpha, \beta},m) $ with $m\geq 2$. On the other hand, in Table \ref{num4} we put the values of $\overline{\lambda}_1^{ev}(\overline p_{\alpha, \beta},m) $ for $m=1,\dots,10$, computed taken $p\equiv1$, and for two suitable choices of $\alpha$ and $\beta$ with $\beta$ satisfying or not the smallness assumption on $\beta$ of Theorem \ref{min}.
\begin{figure}[h!tb]
	{\includegraphics[scale=0.5]{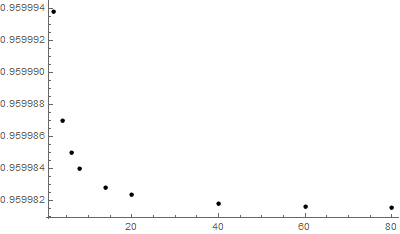}}
	\caption{Plot of $\beta\mapsto \overline{\lambda}_1^{ev}(\overline p_{\alpha, \beta},1) $ with $\ell=\frac{\pi}{150}$ ($\alpha=0.5$).}
	\label{eigplot1}
\end{figure}

\begin{table}[h]\centering
	\scalebox{0.8}{ \begin{tabular}{|c|c|c|c|c|c|c|c|c|c|c|c|c|c|}
			\hline
			 Case&$\overline{\lambda}^{ev}_{1,1}$&$\overline{\lambda}^{ev}_{2,1}$&$\overline{\lambda}^{ev}_{3,1}$&$\overline{\lambda}^{ev}_{4,1}$&$\overline{\lambda}^{ev}_{5,1}$&$\overline{\lambda}^{ev}_{6,1}$&$\overline{\lambda}^{ev}_{7,1}$&$\overline{\lambda}^{ev}_{8,1}$&$\overline{\lambda}^{ev}_{9,1}$&$\overline{\lambda}^{ev}_{10,1}$\\
			\hline
			\hline
			$p\equiv1$&\textbf{0.960009}&15.3610&77.767&245.798&600.145&1244.59&2306.05&3934.57&6303.42&9609.09\\
			\hline
			$\alpha=0.5$, $\beta=1.5$&\textbf{0.959999}&15.3599&77.759&245.755&599.982&1244.10&2304.82&3931.85&6297.92&9598.78\\
			\hline
			$\alpha=0.5$, $\beta=20$&\textbf{0.959982}&15.3589&77.747&245.688&599.724&1243.34&2302.88&3927.53&6289.17&9582.33\\
			\hline
	\end{tabular}}
	\vspace{3mm}
	\caption{The eigenvalues $\overline{\lambda}^{ev}_{m,1}:=\overline{\lambda}_1^{ev}(\overline p_{\alpha, \beta},m)$ with $m=1,\dots,10$ and $\ell=\frac{\pi}{150}$.}
	\label{num4}
\end{table}

All the numerical experiments performed suggest that
$$\text{the map }\beta \mapsto  \overline{\lambda}_1^{ev}(\overline p_{\alpha, \beta},m)  \text{ is decreasing} \quad \text{and } \quad   \overline{\lambda}_1^{ev}(\overline p_{\alpha, \beta},m)  \geq (m-1)^4 \text{ for all } \beta>1\,$$
and the trend does not change varying $\ell$ and $\alpha$. In particular, the above lower bound for $\overline{\lambda}_1^{ev}(\overline p_{\alpha, \beta},m) $ does not depend on $\beta$ and, jointly with the fact that $\overline{\lambda}_1^{ev}(\overline p_{\alpha, \beta},m)<m^4$, supports the conjecture that \
$$
\text{the map }m\mapsto \overline{\lambda}_1^{ev}(\overline p_{\alpha, \beta},m)  \text{ is increasing}
$$
for any $\beta>1$, hence the assumption  $\beta\gg 16(1-\sigma^2)$ of Theorem \ref{min}  seems a merely technical condition.
\par

\section{A positivity preserving property} \label{pppproof}


In this section we state and prove some results about a positivity preserving property for the fourth order differential operator
\begin{equation} \label{eq:Lm}
L_m \varphi=\varphi''''-2m^2\varphi''+m^4 \varphi \, , \qquad m\in \N_+ \, , \ \varphi:[-\ell,\ell]\to \R \, ,
\end{equation}
subject to the boundary conditions:
\begin{equation*}
\begin{cases}
\varphi''(\pm\ell)-\sigma m^2\varphi(\pm\ell)=0 & \qquad  \\
\varphi'''(\pm\ell)-(2-\sigma)m^2\varphi'(\pm\ell)=0\,.& \qquad \,
\end{cases}
\end{equation*}
As in Section \ref{s:main} we fix $0<\sigma<\frac{1}{2}$. These results have their own independent interest and will be exploited in the proofs of Section \ref{spectrum}.

For every $m\in \N_+$, it will be convenient to consider the following scalar product in $H^2(-\ell, \ell)$:
\begin{equation}\label{mprodotto}
\langle\varphi,\phi\rangle_m:=\int_{-\ell}^{\ell}\left( \varphi'' \phi''+2m^2(1-\sigma) \varphi' \phi'-\sigma m^2 (\varphi''\phi+\varphi\phi'')+m^4 \varphi \phi\right) \,dy \,
\end{equation}
which defines an equivalent norm in $H^2(-\ell, \ell)$ that we will denote by  $\lv \phi\rv_{m}^2  =(\phi,\phi)_m$.

\begin{theorem}\label{ppp}
	Let $m\geq 1$ be an integer and let $f\in L^2(-\ell, \ell)$. Furthermore, assume that $w\in H^2(-\ell, \ell)$ is a weak solution to the problem
	\begin{equation}\label{linODE}
	\begin{cases}
	w''''(y)-2m^2w''(y)+m^4w(y)=f (y) & \qquad y\in (-\ell,\ell) \\
	w''(\pm\ell)-\sigma m^2w(\pm\ell)=0 & \qquad  \\
	w'''(\pm\ell)-(2-\sigma)m^2w'(\pm\ell)=0\,.& \qquad \,
	\end{cases}
	\end{equation}
	namely
	\begin{equation}\label{eqppp}
	\langle w,\phi\rangle_m=\int_{-\ell}^{\ell} f \phi \quad \text{ for all } \phi\in H^2(-\ell, \ell)\,,
	\end{equation}
	where $\langle \cdot,\cdot\rangle_m$ is defined in \eqref{mprodotto}.
	Then the following implication holds
	$$f\geq 0\text{ in } (-\ell, \ell)\,(f\not \equiv 0)\quad \Rightarrow \quad w(y)>0 \text{ in } [-\ell, \ell]\,.$$
	Hence, the operator  $L_m$ defined in \eqref{eq:Lm},  under the boundary conditions in \eqref{linODE}, satisfies the positivity preserving property.
\end{theorem}

As a consequence of Theorem \ref{ppp} we have

\begin{corollary}\label{pppweak}
	Let $m\geq 1$, be an integer. Furthermore, set $\mathcal{K}:=\{\phi\in H^2(-\ell, \ell): \phi \geq 0 \text{ in }(-\ell, \ell)\}$ and assume that $w\in H^2(-\ell, \ell)$ satisfies
	\begin{equation}\label{ineqweak}
	\langle w,\phi\rangle_m \leq 0 \quad \text{ for all } \phi\in\mathcal{K} \,,
	\end{equation}
	where $\langle \cdot,\cdot\rangle_m$ is defined in \eqref{mprodotto}.
	Then
	$$\text{either}\quad w \equiv 0 \quad  \text{ or } \quad w<0 \text{ in  } (-\ell, \ell)\,.$$
	
\end{corollary}

\begin{proof}
	Let $f\in \mathcal{K}$ and let $\phi_{f}$ be the unique solution to
	$$\langle\phi_f,\psi\rangle_m=\int_{-\ell}^{\ell} f\psi\,dy \quad \text{ for all } \psi\in  H^2(-\ell, \ell)\,.$$
	By Theorem \ref{ppp}, $\phi_{f}\in \mathcal{K}$. Inserting $\phi_f$ in \eq{ineqweak} we infer
	$$\int_{-\ell}^{\ell} f w\,dy =\langle w,\phi_f\rangle_m \leq 0 \quad \text{ for all } f\in \mathcal{K}\,.$$
	Hence, $w \leq 0$ in $(-\ell, \ell)$. By contradiction, assume that $w \not < 0$ in $(-\ell, \ell)$. Then, if $Z:=\{y\in (-\ell, \ell): w(y)= 0\}$, we have that the characteristic function of $Z$ satisfies $\chi_Z\geq 0$ and $\chi_Z \not \equiv 0$. Let now $\phi_0\in H^2(-\ell, \ell)$ satisfy
	$$\langle\phi_0,\psi\rangle_m=\int_{-\ell}^{\ell} \chi_Z \psi\,dy \quad \text{ for all } \psi\in  H^2(-\ell, \ell)\,.$$
	Since, by elliptic regularity, $\phi_0\in C^3([-\ell, \ell])$ and, by Theorem \ref{ppp}, $\phi_0>0$ in $[-\ell, \ell]$,
	we deduce that for every $\phi\in  H^2(-\ell, \ell)$ there exist $t_1\leq 0\leq t_0$: $\phi +t_0\phi_0\geq 0$ and $\phi+t_1\phi_0\leq 0$ in $[-\ell, \ell]$.
	Furthermore, by definition of $\phi_0$ we have
	$$\langle\phi_0,w\rangle_m=\int_{-\ell}^{\ell} \chi_Z w\,dy=0\,.$$ Combining this with \eq{ineqweak}, we deduce
	$$0\geq \langle \phi+t_0\phi_0,w\rangle_m=\langle\phi,w\rangle_m $$
	and
	$$0\leq \langle \phi+t_1\phi_0,w\rangle_m=\langle\phi,w\rangle_m \,.$$
	Namely,
	$$\langle\phi,w\rangle_m=0 \quad \text{ for all } \phi\in  H^2(-\ell, \ell)\,.$$
	Taking $\phi=w$ in the above inequality we conclude $w \equiv 0$ in $(-\ell, \ell)$ and the proof follows.
\end{proof}

We conclude this section with the proof of Theorem \ref{ppp}.

\medskip

\textbf{Proof of Theorem \ref{ppp}.}

The proof follows by a direct inspection of the sign of the unique solution to \eq{eqppp}.
First we note that, for $m\geq 1$ fixed and $f\in L^2(-\ell, \ell)$, all solutions to the equation
$$w''''(y)-2m^2w''(y)+m^4w(y)=\overline f \quad \text{in } \mathcal{D}'(\R)\,,$$
where $\overline f$ denotes the trivial extension of $f$ to $\R$, write
$$w(y)=c_1 \cosh(my)+c_2 \sinh(my)+c_3 y\cosh(my) +c_4y\sinh(my)+w_p(y)\,,$$
with $c_1,c_2,c_3,c_4 \in \R$ and
$$w_p(y)=(q_m*\overline f)(y)=\int_{-\infty}^{+\infty}q_m(t) \overline f(y-t)\,dt$$
where
$$q_m(y)=\frac{(1+m|y|)e^{-m|y|}}{4m^3}\,.$$
Exploiting the regularity of $q_m$, it follows that all the above solutions belong to $C^3(\R)$ (the regularity can be improved by increasing the regularity of $f$);
the thesis can be reached proving that
\begin{equation}
\widetilde{w}(y)=c_1\cosh(my)+c_2 \sinh(my)+c_3 y\cosh(my)+c_4y\sinh(my)>0
\end{equation}
since $w_p(y)\geq 0$.\par If we fix the constants $c_1,c_2,c_3,c_4 \in \R$ in such a way that:
\begin{equation*}
\begin{cases}
w''(\pm\ell)-\sigma m^2w(\pm\ell)=0 & \qquad  \\
w'''(\pm\ell)-(2-\sigma)m^2w'(\pm\ell)=0\,,& \qquad \,
\end{cases}
\end{equation*}
then the restriction of $w$ to $[-\ell, \ell]$, that we will still denote with $w$, is the unique solution to \eq{eqppp}.
More precisely, by imposing the above conditions we obtain the system
\small{\begin{equation*}
	\begin{cases}
	(c_1m^2+2c_4m) \cosh(m\ell)+(c_2 m^2+2c_3 m) \sinh(m\ell)+c_3m^2\ell \cosh(m\ell) +c_4m^2 \ell \sinh(m \ell)+w''_p(\ell)=\\
	\sigma m^2[c_1 \cosh(m\ell)+c_2 \sinh(m \ell)+c_3 \ell\cosh(m\ell) +c_4\ell\sinh(m\ell)+w_p(\ell)]\\
	(c_1m^2+2c_4m) \cosh(m\ell)-(c_2 m^2+2c_3 m) \sinh(m\ell)-c_3m^2\ell \cosh(m\ell) +c_4m^2 \ell \sinh(m \ell)+w''_p(-\ell)=\\
	\sigma m^2[c_1 \cosh(m\ell)-c_2 \sinh(m \ell)-c_3 \ell\cosh(m\ell) +c_4\ell\sinh(m\ell)+w_p(-\ell)]\\
	(c_2m^3+3c_3m^2) \cosh(m\ell)+(c_1 m^3+3c_4 m^2) \sinh(m\ell)+c_4m^3\ell \cosh(m\ell) +c_3m^3 \ell \sinh(m \ell)+w'''_p(\ell)=\\
	-(\sigma-2) m^2[(c_2m+c_3) \cosh(m\ell)+(c_1m+c_4) \sinh(m \ell)+c_4 m \ell\cosh(m\ell) +c_3m\ell\sinh(m\ell)+w_p'(\ell)]\\
	(c_2m^3+3c_3m^2) \cosh(m\ell)-(c_1 m^3+3c_4 m^2) \sinh(m\ell)-c_4m^3\ell \cosh(m\ell) +c_3m^3 \ell \sinh(m \ell)+w'''_p(-\ell)=\\
	-(\sigma-2) m^2[(c_2m+c_3) \cosh(m\ell)-(c_1m+c_4) \sinh(m \ell)-c_4 m \ell\cosh(m\ell) +c_3m\ell\sinh(m\ell)+w_p'(-\ell)]\,
	\end{cases}
	\end{equation*}}
which decouples in the following two systems
\small{\begin{equation*}
	\begin{cases}
	c_1[2m^2(1-\sigma) \cosh(m\ell)]+c_4[ 4 m\cosh(m\ell)+2m^2(1-\sigma) \ell \sinh(m \ell)]=\\
	\sigma m^2[w_p(\ell)+w_p(-\ell)]-[w_p''(\ell)+w_p''(-\ell)]\\
	c_1[2m^3(\sigma-1) \sinh(m\ell)]+c_4[ 2 m^2(\sigma+1)\sinh(m\ell)+2m^3(\sigma-1) \ell \cosh(m \ell)]=\\
	-(\sigma-2) m^2[w'_p(\ell)-w_p'(-\ell)]-[w_p'''(\ell)-w_p'''(-\ell)]
	\end{cases}
	\end{equation*}}
\small{\begin{equation*}
	\begin{cases}
	c_2[2m^2(1-\sigma) \sinh(m\ell)]+c_3[ 4 m\sinh(m\ell)+2m^2(1-\sigma) \ell \cosh(m \ell)]=\\
	\sigma m^2[w_p(\ell)-w_p(-\ell)]-[w_p''(\ell)-w_p''(-\ell)]\\
	c_2[2m^3(\sigma-1) \cosh(m\ell)]+c_3[ 2 m^2(\sigma+1)\cosh(m\ell)+2m^3(\sigma-1) \ell \sinh(m \ell)]=\\
	-(\sigma-2) m^2[w'_p(\ell)+w_p'(-\ell)]-[w_p'''(\ell)+w_p'''(-\ell)]\,.
	\end{cases}
	\end{equation*}}
By setting
$$F_m(\ell):=(3+\sigma)\sinh(m\ell)\cosh(m\ell)-m \ell (1-\sigma)>0\,,$$
$$\overline F_m(\ell):=(3+\sigma)\sinh(m\ell)\cosh(m\ell)+m \ell (1-\sigma)>0\,,$$
$$A_m(\ell):=(1+\sigma)\sinh(m\ell)-(1-\sigma)m\ell \cosh(m\ell)\,, \quad B_m(\ell):= 2\cosh(m\ell)+(1-\sigma)m\ell\sinh(m\ell)\,,$$
$$\overline A_m(\ell):=(1+\sigma)\cosh(m\ell)-(1-\sigma)m\ell \sinh(m\ell)\,,\quad \overline B_m(\ell):= 2\sinh(m\ell)+(1-\sigma)m\ell\cosh(m\ell)\,,$$
$$V_m(\ell):=\sigma m^2 w_p(\ell)-w''_p(\ell)\,, \quad W_m(\ell):=(\sigma-2) m^2 w_p'(\ell)+w'''_p(\ell)\,,$$

$$V_m(-\ell):=\sigma m^2 w_p(-\ell)-w''_p(-\ell)\,, \quad W_m(-\ell):=(\sigma-2) m^2 w_p'(-\ell)+w'''_p(-\ell)\,,$$

the solutions to the above systems write
$$c_1=\frac{m A_m(\ell)[V_m(\ell)+V_m(-\ell)]+B_m(\ell)[W_m(\ell)-W_m(-\ell)]}{2m^3(1-\sigma) F_m(\ell)}$$
$$c_2=\frac{m \overline A_m(\ell)[V_m(\ell)-V_m(-\ell)]+\overline B_m(\ell)[W_m(\ell)+W_m(-\ell)]}{2m^3(1-\sigma) \overline F_m(\ell)}$$
$$c_3=\frac{m \cosh(m\ell)[V_m(\ell)-V_m(-\ell)]-\sinh(m\ell)[W_m(\ell)+W_m(-\ell)]}{2m^2 \overline F_m(\ell)}$$
$$c_4=\frac{m \sinh(m\ell)[V_m(\ell)+V_m(-\ell)]-\cosh(m\ell)[W_m(\ell)-W_m(-\ell)]}{2m^2 F_m(\ell)}\,.$$

By exploiting the symmetry of $q_m$, for $i=0$ and $i=2$, we have
$$w_p^{(i)}(\ell)=\int_{0}^{2\ell}q_m^{ (i) }(t)  f(\ell-t)\,dt\,, \quad w_p^{(i)}(-\ell)=\int_{0}^{2\ell}q_m^{ (i) }(t)  f(-\ell+t)\,dt\,,$$
while, for $i=1$ and $i=3$, we have
$$w_p^{(i)}(\ell)=\int_{0}^{2\ell}q_m^{ (i) }(t) f(\ell-t)\,dt\,, \quad w_p^{(i)}(-\ell)=-\int_{0}^{2\ell}q_m^{ (i) }(t)  f(-\ell+t)\,dt\,,$$
where $w_p^{(i)}$ and $q_m^{ (i)}$ denotes, respectively, the $i$-th derivate of $w_{p}$ and of $q_m$. Hence,
$$V_m(\ell)=\int_{0}^{2\ell}\frac{e^{-mt}}{4m}(1+\sigma-mt(1-\sigma))  f(\ell-t)\,dt\,,\quad \quad W_m(\ell)=\int_{0}^{2\ell}\frac{e^{-mt}}{4}(2+mt(1-\sigma)) f(\ell-t)\,dt\,$$
and
$$V_m(-\ell)=\int_{0}^{2\ell}\frac{e^{-mt}}{4m}(1+\sigma-mt(1-\sigma))  f(-\ell+t)\,dt\,,\quad \quad W_m(-\ell)=-\int_{0}^{2\ell}\frac{e^{-mt}}{4}(2+mt(1-\sigma)) f(-\ell+t)\,dt\,.$$
First of all we study the sign of the coefficients $c_1$ and $c_4$. Since $F_m(\ell)>0$, $c_1$ has the same sign of
$$m A_m(\ell)[V_m(\ell)+V_m(-\ell)]+B_m(\ell)[W_m(\ell)-W_m(-\ell)]=$$
$$
\int_{0}^{2\ell}\frac{e^{-mt}}{4}\,\left[A_m(\ell)(1+\sigma-m(1-\sigma)t)+B_m(\ell)(2+m(1-\sigma)t) \right] \left[f(\ell-t)+ f(-\ell+t)\right]\,dt
$$

\begin{equation}\label{c11}
=[\sinh(m\ell)((1+\sigma)^2+2m\ell(1-\sigma))+\cosh(m\ell)(4-(1-\sigma^2)m\ell)]\int_{0}^{2\ell}\frac{e^{-mt}}{4}\,[f(\ell-t)+ f(-\ell+t)]\,dt+
\end{equation}
\begin{equation}\label{c12}
+m(1\!-\!\sigma)\!\left[2\cosh(m\ell)\!-\!(\!1+\!\sigma)\sinh(m\ell)\!+\!(1\!-\!\sigma)m\ell(\cosh(m\ell)\!+\!\sinh(m\ell))\right]
\int_{0}^{2\ell}\dfrac{e^{-mt}t}{4}\, [f(\ell-t)+ f(-\ell+t)]\,dt\,.
\end{equation}

We observe that
\begin{equation}\label{diseq1}
2\cosh(z)>(1+\sigma)\sinh(z)
\end{equation}
for $z>0$ and for all $\sigma\in (0,1/2)$, implying that \eqref{c12} is positive; about  the sign of \eqref{c11} we introduce the map
$$
z\mapsto g(z):=\sinh(z)((1+\sigma)^2+2z(1-\sigma))+\cosh(z)(4-(1-\sigma^2)z)
$$
and we compute its derivative
$$g'(z)=2\sigma(1+\sigma)\cosh(z)+2(3-\sigma)\sinh(z)+z(1-\sigma)[2\cosh(z)-(1+\sigma)\sinh(z)].$$
Thanks to \eqref{diseq1}, for $z>0$ we obtain $g'(z)>0$ so that $g(z)$ is always positive ($g(0)=4$)  and in particular
$c_1>0$.  The sign of $c_4$ depends on
$$m \sinh(m\ell)[V_m(\ell)+V_m(-\ell)]-\cosh(m\ell)[W_m(\ell)-W_m(-\ell)]=$$
\begin{equation*}
\begin{split}
\int_{0}^{2\ell}\frac{e^{-mt}}{4}\big[(1+\sigma)\sinh(m\ell)-2\cosh(m\ell)-mt(1-\sigma)(\sinh(m\ell)+\cosh(m\ell))\big][f(\ell-t)+ f(-\ell+t)]\,dt
\end{split}
\end{equation*}
that, applying again \eqref{diseq1}, gives $c_4<0$ for all $\sigma\in (0,1/2)$ and $m\ell>0$.\par
For our purposes we need to compare the absolute value of $c_4$ and $c_3$; since the sign of $c_3$ is not known a priori,
we study the sign of $2m^2(|c_4|\pm c_3)$, i.e.
\begin{align*}
& \int_{0}^{2\ell}\frac{e^{-mt}}{4 F_m(\ell)}\big[2\cosh(m\ell)-(1+\sigma)\sinh(m\ell)+mt(1-\sigma)(\sinh(m\ell)+\cosh(m\ell))\big][f(\ell-t)+ f(-\ell+t)]\,dt \\
&\pm\int_{0}^{2\ell}\frac{e^{-mt}}{4 \overline{F}_m(\ell)}\big[(1+\sigma)\cosh(m\ell)-2\sinh(m\ell)-mt(1-\sigma)(\sinh(m\ell)+\cosh(m\ell))\big]
[f(\ell-t)-f(-\ell+t)]\,dt  \, .
\end{align*}
Recalling that $0<F_m(\ell)<\overline{F}_m(\ell)$, we obtain the positivity of
\begin{align*}
& m(1\!\!-\!\!\sigma)[\sinh(m\ell)\!+\!\cosh(m\ell)]  \bigg\{\!\bigg[\!\frac{1}{F_m(\ell)}\!\!\mp\!\!\frac{1}{\overline{F}_m(\ell)}\bigg]\!\!\int_{0}^{2\ell}
\frac{e^{-mt}t}{4}f(\ell\!\!-\!\!t)\,dt
\!+\!\bigg[\frac{1}{F_m(\ell)}\!\!\pm\!\! \frac{1}{\overline{F}_m(\ell)}\bigg]\!\!\int_{0}^{2\ell}\dfrac{e^{-mt}t}{4}f(-\ell+t)\,dt\bigg\};
\end{align*}
thus $2m^2(|c_4|\pm c_3)>0$ if
\begin{equation*}
\frac{2\cosh(m\ell)-(1+\sigma)\sinh(m\ell)}{F_m(\ell)}\pm\frac{(1+\sigma)\cosh(m\ell)-2\sinh(m\ell)}{\overline{F}_m(\ell)}>0;
\end{equation*}
the achievement follows from the positivity  of $\big(\cosh(z)\mp\sinh(z)\big)\big(2\pm(1+\sigma)\big)$ for all $z>0$ and $\sigma\in(0,1/2)$.

Fixed $m\geq1$, we set
$$
\widetilde{\psi}(t):=m\,\widetilde{w}(t/m)=c_1m\cosh t+c_2m\sinh t+c_3t\cosh t+c_4t\sinh t
$$
and we focus on the qualitative behaviour of $\widetilde{\psi}$ where, from above, $c_1>0$, $c_4<0$ and $c_4<c_3<-c_4$. Clearly, $\widetilde{\psi}(t)$
is continuous and differentiable on $\mathbb{R}$, moreover
\begin{equation*}
\widetilde{\psi}(0)=m\,c_1>0\,;\quad \widetilde{\psi}(t)\sim \dfrac{c_3\pm c_4}{2}te^{|t|}\rightarrow-\infty \hspace{2mm}{\rm for }
\hspace{2mm}t\rightarrow\pm\infty\,. \quad
\end{equation*}
This fact implies that $\widetilde{\psi}(t)$ has at least two zeros of opposite sign on $\mathbb{R}$; we prove now that $\widetilde{\psi}(t)$
has exactly two distinct zeros on $\mathbb{R}$.

We know that $\widetilde{\psi}(t)=0$ if and only if
$$
\alpha(t):=(c_2m+c_4t)\tanh t+c_3t+c_1m=0 \, .
$$
Computing $\alpha'(t)=\frac{1}{2\cosh^2(t)}(2c_3\cosh^2(t)+c_4\sinh(2t)+2c_4t+2c_2m)$ we observe that
\begin{equation}\label{zeros}
\exists!\hspace{1mm} \overline{t}\in\mathbb{R}: \alpha'(\overline{t})=0 \, .
\end{equation}
This follows because $\beta(t):=2c_3\cosh^2(t)+c_4\sinh(2t)+2c_4t+2c_2m$ is always decreasing on $\mathbb{R}$; indeed $c_4<0$, $|c_4|>|c_3|$ so that $\beta'(t)=2(c_3\sinh(2t)+c_4\cosh(2t)+c_4)<0$. Moreover $\beta(t)\sim\dfrac{c_3\pm c_4}{2}e^{2|t|}\rightarrow\mp\infty$
for $t\rightarrow\pm\infty$. Now let us suppose for contradiction that $\widetilde{\psi}(t)$ has more than two zeros on $\mathbb{R}$, for instance it has 3
distinct zeros $t_1<t_2<t_3$; this implies that $\alpha(t)$ has 3 distinct zeros, then, the Rolle's Theorem applied to $\alpha(t)$ in the intervals $[t_1,t_2]$
and $[t_2,t_3]$ ensures  the existence of  at least two points in which $\alpha'(t)=0$ on $\mathbb{R}$ and this contradicts \eqref{zeros}. Hence, $\widetilde{\psi}$, and
in turn also $\widetilde{w}$, has exactly two zeros of opposite sign on $\mathbb{R}$.

Since $\widetilde{w}(y)$ has exactly two zeros of opposite sign on $\mathbb{R}$ and $\widetilde{w}(0)>0$, if we prove that $\widetilde{w}(\pm\ell)>0$
the thesis follows. To this aim we study the sign of $\widetilde{w}(\pm\ell)=c_1\cosh(m\ell)\pm c_2\sinh(m\ell)\pm c_3\ell\cosh(m\ell)+c_4\ell\sinh(m\ell)$,
in particular we consider
\begin{align*}
& 2m^2\widetilde{w}(\ell)=\int_{0}^{2\ell}\frac{e^{-mt}}{4m}[C_m(\ell)f(\ell-t)+\overline{C}_m(\ell)f(-\ell+t)]\,dt
+\int_{0}^{2\ell}\frac{e^{-mt}t}{4}[D_m(\ell)f(\ell-t)+\overline{D}_m(\ell)f(-\ell+t)]\,dt \\[8pt]
& 2m^2\widetilde{w}(-\ell)=\int_{0}^{2\ell}\frac{e^{-mt}}{4m}[\overline{C}_m(\ell)f(\ell-t)+C_m(\ell)f(-\ell+t)]\,dt
+\int_{0}^{2\ell}\frac{e^{-mt}t}{4}[\overline{D}_m(\ell)f(\ell-t)+D_m(\ell)f(-\ell+t)]\,dt
\end{align*}
where
\begin{align*}
& C_m(\ell)\!\!=\!\!\frac{4}{1-\sigma}\bigg(\frac{\cosh^2(m\ell)}{F_m(\ell)}\!+\!\frac{\sinh^2(m\ell)}{\overline{F}_m(\ell)}\bigg)
\!+\!\frac{(1+\sigma)^2}{2(1-\sigma)}\sinh(2m\ell)\bigg(\frac{1}{F_m(\ell)}\!+\!\frac{1}{\overline{F}_m(\ell)}\bigg)
\!-\!m\ell(1+\sigma)\bigg(\frac{1}{F_m(\ell)}\!-\!\frac{1}{\overline{F}_m(\ell)}\bigg)\\
& D_m(\ell)=2\bigg(\frac{\cosh^2(m\ell)}{F_m(\ell)}+\frac{\sinh^2(m\ell)}{\overline{F}_m(\ell)}\bigg)-\frac{1+\sigma}{2}\sinh(2m\ell)\bigg(\frac{1}{F_m(\ell)}
+\frac{1}{\overline{F}_m(\ell)}\bigg)+m\ell(1-\sigma)\bigg(\dfrac{1}{F_m(\ell)}-\frac{1}{\overline{F}_m(\ell)}\bigg) \\
& \overline{C}_m(\ell)\!\!=\!\!\frac{4}{1-\sigma}\bigg(\frac{\cosh^2(m\ell)}{F_m(\ell)}\!-\! \frac{\sinh^2(m\ell)}{\overline{F}_m(\ell)}\bigg)
+\frac{(1+\sigma)^2}{2(1-\sigma)}\sinh(2m\ell)\bigg(\frac{1}{F_m(\ell)}\!-\!\frac{1}{\overline{F}_m(\ell)}\bigg)\!-\!m\ell(1+\sigma)\bigg(\frac{1}{F_m(\ell)}
\!+\!\frac{1}{\overline{F}_m(\ell)}\bigg) \\
&\overline{D}_m(\ell)=2\bigg(\frac{\cosh^2(m\ell)}{F_m(\ell)}-\frac{\sinh^2(m\ell)}{\overline{F}_m(\ell)}\bigg)
-\frac{1+\sigma}{2}\sinh(2m\ell)\bigg(\frac{1}{F_m(\ell)}-\frac{1}{\overline{F}_m(\ell)}\bigg)
+m\ell(1-\sigma)\bigg(\frac{1}{F_m(\ell)}+\frac{1}{\overline{F}_m(\ell)}\bigg).
\end{align*}
The final part of the proof is devoted to prove that the coefficients $C_m(\ell), D_m(\ell), \overline{C}_m(\ell)$ and $\overline{D}_m(\ell)$ are positive.
We recall that
$$
\frac{1}{F_m(\ell)}+\dfrac{1}{\overline{F}_m(\ell)}=\dfrac{(3+\sigma)\sinh(2m\ell)}{F_m(\ell)\overline{F}_m(\ell)}>0
\hspace{10mm}
\frac{1}{F_m(\ell)}-\dfrac{1}{\overline{F}_m(\ell)}=\dfrac{2m\ell(1-\sigma)}{F_m(\ell)\overline{F}_m(\ell)}>0 \, ,
$$
and we introduce four maps related respectively to the previous coefficients
\begin{align*}
& z\mapsto p(z):=\dfrac{2(3+\sigma)}{1-\sigma}\sinh(2z)\cosh(2z)+4z+\dfrac{(1+\sigma)^2(3+\sigma)}{2(1-\sigma)}\sinh^2(2z)-2(1-\sigma^2)z^2 \\[8pt]
& z\mapsto q(z):=\dfrac{3+\sigma}{2}\sinh(2z)[2\cosh(2z)-(1+\sigma)\sinh(2z)]+2(1-\sigma)z+2(1-\sigma)^2z^2 \\[8pt]
& z\mapsto r(z):=\dfrac{2(3+\sigma)}{1-\sigma}\sinh(2z)+z[4\cosh(2z)-2(1+\sigma)\sinh(2z)] \\[8pt]
& z\mapsto s(z):=(3+\sigma)\sinh(2z)+(1-\sigma)z[2\cosh(2z)-(1+\sigma)\sinh(2z)]+(1-\sigma)(3+\sigma)z\sinh(2z).
\end{align*}
Thanks to \eqref{diseq1} $q(z)$, $r(z)$ and $s(z)$ are always positive for $z>0$ and for all $\sigma\in(0,1/2)$.
The same conclusion holds for the maps $p(z)$, due to the following inequality
$$
\frac{(1+\sigma)^2(3+\sigma)}{1-\sigma}\sinh^2(2z)>3\sinh^2(2z)>(1-\sigma^2)(2z)^2 \, .
$$
This completes the proof.


\section{Proof of Theorem \ref{thm-exist-qual} and Proposition \ref{comparison_esy}} \label{s:proofs-1}

\subsection{Proof of Theorem \ref{thm-exist-qual}}

We start with the existence issue.

\begin{lemma}\label{exist}
The infimum in \eqref{CP} is achieved.
\end{lemma}

\begin{proof}
Let $\{ p_m \}_m \subset P_{\alpha,\beta}$ be a minimizing sequence for $ \lambda_{\alpha,\beta}$, i.e.

$$
\lambda_1(p_m)= \lambda_{\alpha,\beta}+o(1) \quad \text{as } m \rightarrow \infty\,.
$$
Furthermore, let $u_{p_m} \in H^2_*(\Omega)$ be a (normalized) eigenfunction to $\lambda_1(p_m)$, namely $\lambda_1(p_m) = ||u_{p_m}||^2_{H^2_*(\Omega)}$
and $\int_{\Omega} p_m \, u_{p_m}^2 \, dx\,dy = 1.$ This immediately implies
$||u_{p_m}||_{H^2_*} \leq C,$ for some positive constant $C.$ Therefore, using the compact embedding of $H^2_*(\Omega) \hookrightarrow L^2(\Omega),$
 we can extract two subsequences, still denoted
by $u_{p_m},$ such that

$$
u_{p_m} \rightharpoonup  \overline{u} \quad \mbox{in} \ H^2_*(\Omega) \quad \mbox{as} \ m \rightarrow \infty,
$$

$$
u_{p_m} \rightarrow \overline{u} \quad \mbox{in} \ L^2(\Omega) \quad \mbox{as} \ m \rightarrow \infty.
$$

Moreover, $p_m \in P_{\alpha,\beta}$ implies $||p_{m}||_{L^{\infty}(\Omega)} \leq \beta$ and therefore up to
a subsequence, $p_m \rightharpoonup \overline{p}$ in $L^2(\Omega)$ as $m \rightarrow \infty,$ so that also $\overline p$ satisfies the integral condition in \eqref{eq:famiglia}.
Since strongly closed convex sets are weakly closed, we infer that $\alpha \leq \overline{p}\leq \beta$ a.e. in $\Omega$.
Hence, $\overline{p} \in P_{\alpha, \beta}$. On the other hand, we obtain

\begin{align*}
\left| \int_{\Omega}( p_m \, u_{p_m}^2 - \overline{p} \,\overline{u}^2 ) dx\, dy \right|
& = \left|  \int_{\Omega}  p_m ( u_{p_m}^2 -  \overline{u}^2 ) \, dx\, dy +  \int_{\Omega} \overline{u}^2( p_m - \overline{p}) \, dx\, dy  \right| \notag \\
& \leq \beta \int_{\Omega} |(u_{p_m}-\overline{u})(u_{p_m}+\overline{u}) | \, dx\, dy + \| \overline{u}\|_{L^{\infty}(\Omega)}
\int_{\Omega} |\overline{u}|\, |p_m - \overline{p}| \, dx\, dy\\
&=2\beta \, \| \overline{u}\|_{ L^2(\Omega)}\,\|u_{p_m}- \overline{u}\|_{ L^2(\Omega)}+ o(1)=o(1) \quad \mbox{as} \ m \rightarrow \infty\,, \notag
\end{align*}
where we have exploited the fact that $H^2_*(\Omega) \subset L^{\infty}(\Omega)$ since $\Omega$ is a planar domain. In particular, we conclude that
$\int_{\Omega} \overline{p} \, \overline{u}^2 \, dx\, dy = 1.$
Furthermore,

$$
\lambda_1(\overline{p}) \leq ||\overline{u}||^2_{H^2_*} \leq \liminf_{m \rightarrow \infty} ||u_{p_m}\|_{H^2_*}^2 = \lambda_{\alpha,\beta} \, .
$$
Whence

$$
\lambda_{\alpha,\beta} \leq \lambda_1(\overline{p}) = ||\overline{u}||^2_{H^2_*} \leq \lambda_{\alpha,\beta}.
$$
Therefore, the couple $(\overline{p},\overline{u})$ is an optimal pair and $u_{\overline{p}}=\overline{u}$; this completes the proof.

\end{proof}

To problem \eqref{CP} we associate the following infimum problem

\begin{equation}\label{ACP}
\Lambda_{\alpha,\beta} := \inf_{\eta \in N_{\alpha,\beta}} \, \min_{u \in H^2_*(\Omega)\setminus \{0\}}
\frac{||u||_{H^2_*(\Omega)}^2 \, + \,  \lambda_{\alpha,\beta}(\beta-\alpha) \int_{\Omega} \eta \, u^2 \, dx\, dy}{\int_{\Omega} u^2 \, dx\, dy},
\end{equation}
where $\lambda_{\alpha,\beta}$ is as in \eqref{CP} and

$$
N_{\alpha,\beta} = \left\{ \eta\in L^\infty(\Omega):\, 0 \leq \eta \leq 1 \ \text{a.e. in } \Omega \ \mbox{and} \ \int_{\Omega} \eta \, dx\, dy
= \frac{\beta -1}{\beta-\alpha} |\Omega| \right\} \, .
$$

The proof of Lemma \ref{exist} with minor changes  shows that also problem \eqref{ACP} admits an optimal pair
$(\overline{\eta},u_{\overline \eta}) \in N_{\alpha,\beta}  \times H^2_*(\Omega)$. Furthermore, there is an one-to-one correspondence
between problems \eqref{CP} and \eqref{ACP}. Indeed, to any
$\eta \in N_{\alpha,\beta}$ we  can associate $p_{\eta} \in P_{\alpha,\beta}$ by setting

$$
p_{\eta}=\beta-\eta(\beta-\alpha).
$$

Clearly $\alpha \leq p_{\eta} \leq \beta$ and

$$
\int_{\Omega} p_{\eta} \, dx\, dy = \beta |\Omega| - (\beta - \alpha) \int_{\Omega} \eta \, dx\, dy = |\Omega|.
$$
Viceversa to any $p \in P_{\alpha,\beta}$ we can associate $\eta_{p} \in N_{\alpha,\beta}$ by setting

$$
\eta_p = \dfrac{\beta - p}{\beta - \alpha}.
$$
Clearly $0 \leq \eta \leq 1$ and $\int_{\Omega} \eta_{p} \, dx\, dy = \frac{\beta - 1}{\beta - \alpha} |\Omega|.$ Furthermore, we have

\begin{lemma}\label{thm-1}
Let $\lambda_{\alpha, \beta}$ and $\Lambda_{\alpha, \beta}$ be as defined in \eqref{CP} and in \eqref{ACP}. There holds
$$
\Lambda_{\alpha, \beta} = \lambda_{\alpha, \beta}\,\beta.
$$
\end{lemma}

\begin{proof}
We shall prove the lemma in two steps.

\medskip

{\bf Step 1 :} Let $\overline{p} \in P_{\alpha, \beta}$ and $u_{\overline{p}} \in H^2_*(\Omega)$ such that $\lambda_{\alpha, \beta}$ is achieved
for this optimal pair and  let $\eta_{\overline{p}} = \frac{\beta - \overline{p}}{\beta - \alpha} \in N_{\alpha, \beta}$.
Clearly we have

\begin{align*}
\Lambda_{\alpha, \beta}&  \leq \min_{u\in H^2_*(\Omega)\setminus \{0\}}  \, \dfrac{||u||_{H^2_*(\Omega)}^2 + \lambda_{\alpha, \beta} (\beta-\alpha)
\int_{\Omega} \eta_{\overline{p}} \,  u^2 \, dx\, dy}{\int_{\Omega} u^2 \, dx\, dy} \notag  \\[9pt]
& = \min_{u\in H^2_*(\Omega)\setminus \{0\}} \,  \dfrac{ ||u||_{H^2_*(\Omega)}^2 - \lambda_{\alpha, \beta} \int_{\Omega} \overline{p} \, u^2 \, dx\, dy
 + \lambda_{\alpha, \beta}
 \beta \int_{\Omega} u^2 \, dx\, dy }{\int_{\Omega} u^2 \, dx\, dy} \notag \\[9pt]
 & \leq \underbrace{ \frac{||u_{\overline{p}}||_{H^2_*(\Omega)}^2 -  \lambda_{\alpha,\beta}
 \int_{\Omega} \overline{p} \, u_{\overline{p}}^2 \, dx\, dy}{\int_{\Omega} u_{\overline{p}}^2 \, dx\, dy}}_{= 0}
 + \lambda_{\alpha, \beta} \beta=\lambda_{\alpha, \beta}\, \beta.
\end{align*}

\medskip

{\bf Step 2 :} Let now $\eta \in N_{\alpha, \beta}$ and $p_{\eta} \in P_{\alpha, \beta}$
with $\eta = \frac{\beta - p_{\eta}}{\beta - \alpha},$ i.e., $p_{\eta} = \beta - \eta (\beta - \alpha).$
 Then for any $u\in H^2_*(\Omega)\setminus \{0\}$

\begin{align}\label{estimate-2}
\frac{||u||_{H^2_*(\Omega)}^2 + \lambda_{\alpha, \beta} (\beta - \alpha) \int_{\Omega} \eta \,  u^2 \, dx\, dy}{\int_{\Omega} u^2 \, dx\, dy}
=\frac{ ||u||_{H^2_*(\Omega)}^2 - \lambda_{\alpha, \beta} \int_{\Omega} {p_{\eta}} \, u^2 \, dx\, dy +\lambda_{\alpha, \beta}
 \beta \int_{\Omega} u^2 \, dx\, dy }{\int_{\Omega} u^2 \, dx\, dy}.
\end{align}

Since, $p_{\eta} \in P_{\alpha, \beta}$ implies $\lambda_{\alpha, \beta} \leq \frac{||u||_{H^2_*(\Omega)}^2}{\int_{\Omega} p_{\eta}\, u^2 \, dx\, dy}$ for any
$u \in H^2_*(\Omega)\setminus \{0\}$
and $\eta \in N_{\alpha, \beta}$, passing to the infima,  \eqref{estimate-2} yields
$$
\Lambda_{\alpha, \beta} \geq \lambda_{\alpha, \beta}\, \beta.
$$
This completes the proof.
\end{proof}

Finally, we prove that the optimal pair of problem \eqref{ACP} can be characterised as follows

\begin{lemma}\label{Lambda}
Let $(\overline{\eta},\overline u) \in N_{\alpha, \beta} \times H^2_*(\Omega) $
be an optimal pair of problem \eqref{ACP}. Then, $\overline{\eta}$ and $\overline u$ are related as follows
$$
\overline{\eta}(x,y) = \chi_{S_{\overline u}} (x,y) \qquad \text{for a.e. } (x,y)\in \Omega \, ,
$$
where $\chi_{S_{\overline u}}$ is the characteristic function of a set $S_{\overline u}\subset \Omega$ such that
$|S_{\overline u}|=\frac{\beta-1}{\beta -\alpha}\,|\Omega|$
and
$$
S_{\overline u}=\{(x,y)\in \Omega\,: \overline u^2(x,y) \leq t \}
$$
for some $t> 0$.
\end{lemma}

\begin{proof} The proof is along the line of \cite[Proposition~3.3]{CV}. For the sake of completeness we shall
 outline the main ideas.

\textbf{ Step 1.}  Let $u \in H^2_*(\Omega)$ be such that $||u||_2 = 1$ and consider the functional $I : N_{\alpha, \beta} \rightarrow \mathbb{R}$
$$
I(\eta) := \int_{\Omega} \eta \, u^2 \, dx\, dy\,.
$$
We prove that the infimum problem
$$
I_{\alpha, \beta} := \inf_{\eta \in N_{\alpha, \beta}} I(\eta)
$$

admits a solution $\eta = \chi_{S_u},$ where $S_u \subset \Omega$ is such that $|S_u| = \frac{\beta -1}{\beta - \alpha} |\Omega|$ and satisfies one of the following
\begin{equation}\label{S}
S_u \subseteq \{(x,y)\in \Omega \, : u^2(x,y)=0 \} \quad \text{ or }\quad  \{(x,y)\in \Omega\,:u^2(x,y) < t \} \subseteq S_u \subseteq \{(x,y)\in \Omega\,: u^2(x,y)\leq t\},
\end{equation}
where $t$ is defined as
\begin{equation} \label{eq:def-t}
t:= \sup \left\{ s > 0 : |\{(x,y)\in \Omega\,: u^2(x,y) < s \}| < \frac{\beta -1}{\beta-\alpha} |\Omega| \right\}.
\end{equation}
Let $S_u \subset \Omega$ be as above, then $\chi_{S_u} \in N_{\alpha,\beta}$ and one obtains
$$
I_{\alpha, \beta} \leq I(\chi_{S_u}) = \int_{S_u} u^2 \, dx\, dy.
$$
On the other hand we claim that the following inequality holds
$$
I(\eta) \geq I(\chi_{S_u})   \qquad \text{for any } \eta\in N_{\alpha,\beta} \,.
$$
If this is true then one immediately obtain $I_{\alpha, \beta} = I (\chi_{S_u})$ and this concludes the proof of step 1.

We prove the validity of the claim by considering the cases $t>0$ and $t=0$ separately.

If $t>0$, we argue as follows
\begin{align} \label{eq:optimal}
& \int_{\Omega} u^2 (\chi_{S_u} - \eta) \, dx\, dy \\[8pt]
\notag & \qquad =\int_{\{ u^2 < t\}} u^{2} \, (\chi_{S_u}-\eta) \, dx\, dy + \int_{\{ u^2 > t\}} u^{2} \, (\chi_{S_u} - \eta) \, dx\, dy +
\int_{\{ u^2 = t\}} u^{2} \, (\chi_{S_u}-\eta) \, dx\, dy \\[8pt]
\notag & \qquad \leq t \int_{\{ u^2 < t\}} (\chi_{S_u} - \eta) \, dx\, dy - t \int_{\{ u^2 > t\}} \eta \, dx\, dy
+t \int_{\{ u^2 = t\}} \, (\chi_{S_u} - \eta) \, dx\, dy \\[8pt]
& \notag \qquad = t \int_{\Omega} (\chi_{S_u}-\eta) \, dx\, dy=0 \, .
\end{align}
If $t=0$ the proof follows with minor changes.

\textbf{ Step 2.} We prove that if $(\overline \eta, \overline u)$ is an optimal pair as in the statement of the lemma
and if $S_{\overline u}$ is the corresponding set defined according to
Step 1, then $(\chi_{S_{\overline u}},\overline u)$ is still an optimal pair.

Set
$$
S_{\alpha, \beta}:=\left\{S \subset \Omega\,:|S| = \frac{\beta -1}{\beta - \alpha} |\Omega|\right\} \, .
$$
Since $\{ \chi_S : S \in S_{\alpha, \beta} \} \subset N_{\alpha, \beta},$  we have
\begin{align*}
\Lambda_{\alpha,\beta}
\leq \inf_{S \in S_{\alpha, \beta}} \, \min_{u \in H^2_*(\Omega)\setminus \{0\}} \, \frac{||u||_{H^2_*(\Omega)}^2 + \lambda_{\alpha, \beta}(\beta - \alpha)
\int_{\Omega} \chi_{S} \, u^2 \, dx\, dy}{\int_{\Omega} u^2 \, dx\, dy} \, .
\end{align*}
On the other hand, letting $(\overline \eta,\overline u)$ an optimal pair as in the statement of the lemma, from Step 1 we have
$$
||\overline u||_{H^2_*(\Omega)}^2 +\lambda_{\alpha, \beta}(\beta - \alpha)  \int_{\Omega}  \, \overline{\eta}  \, \overline{u}^2 \, dx\, dy
\geq ||\overline u||_{H^2_*(\Omega)}^2 + \lambda_{\alpha, \beta}(\beta - \alpha) \int_{\Omega} \chi_{S_{\overline u}} \, \overline u^2 \, dx\, dy
$$
and therefore
\begin{align*}
\Lambda_{\alpha,\beta}&=\frac{||\overline u||_{H^2_*(\Omega)}^2 +\lambda_{\alpha, \beta}(\beta - \alpha)  \int_{\Omega}  \,
\overline \eta  \, \overline u^2 \, dx\, dy}
{\int_\Omega \overline u^2 \, dx\, dy}\ge
\frac{||\overline u||_{H^2_*(\Omega)}^2 + \lambda_{\alpha, \beta}(\beta - \alpha) \int_{\Omega} \chi_{S_{\overline u}} \, \overline u^2 \, dx\, dy}
{\int_\Omega \overline u^2 \, dx\, dy} \\[8pt]
& \ge\inf_{S \in S_{\alpha, \beta}} \, \min_{u \in H^2_*(\Omega)\setminus \{0\}}
\, \frac{||u||_{H^2_*(\Omega)}^2 + \lambda_{\alpha, \beta}(\beta - \alpha)  \int_{\Omega} \chi_{S} \, u^2 \, dx\, dy}{\int_{\Omega} u^2 \, dx\, dy} \, .
\end{align*}
This proves that
\begin{align*}
\Lambda_{\alpha,\beta}=
\inf_{S \in S_{\alpha, \beta}} \, \min_{u \in H^2_*(\Omega)\setminus \{0\}} \, \frac{||u||_{H^2_*(\Omega)}^2 + \lambda_{\alpha, \beta}(\beta - \alpha)
\int_{\Omega} \chi_{S} \, u^2 \, dx\, dy}{\int_{\Omega} u^2 \, dx\, dy}
\end{align*}
and in particular that $(\chi_{S_{\overline u}},\overline u)$ is an optimal pair.

\textbf{ Step 3.} Let $(\chi_{S_{\overline u}},\overline u)$ be the optimal pair introduced in Step 2 and let $\overline t$ be the number $t$ in \eqref{S}
corresponding to $\overline u$.
Let
$$
A_{\overline t}=\{ (x,y)\in \Omega\,: \overline u^2(x,y) = \overline t\} \, .
$$
We prove that $\overline t>0$ and that $|A_{\overline t}\setminus S_{\overline u}|=0$.

Suppose by contradiction that $\overline t=0$.
Since $\overline u \in H^4(\Omega)$ we can write the Euler-Lagrange equation related to  \eqref{ACP} almost everywhere and we have
$$
\Lambda_{\alpha, \beta} \, \overline u=\Delta^2 \overline u+ \lambda_{\alpha,\beta} (\beta-\alpha) \chi_{S_{\overline u}} \, \overline u=\Delta^2 \overline u \quad \text{a.e. in } \Omega \, . $$
Since $\overline u$ satisfies the partially hinged boundary conditions this means that it must be one of the eigenfunctions listed in Proposition \ref{eigenvalue} which is impossible since the set of zeroes of any of the eigenfunctions of Proposition \ref{eigenvalue} has zero measure thus contradicting the definition
of $S_{\overline u}$ which forces $S_{\overline u}$ to be a set of positive measure. This proves that $\overline t>0$.

Suppose now by contradiction that $|A_{\overline t}\setminus S_{\overline u}|>0$, we have that
$$
\Delta^2 \overline u+ \lambda_{\alpha, \beta} (\beta-\alpha) \chi_{S_{\overline u}} \, \overline u=\Lambda_{\alpha, \beta} \overline u\quad
\text{a.e. in } \Omega \, .
$$
Now, exploiting the fact that $\overline u$ is constant in $A_{\overline t}$ and $\overline t>0$, we infer
$$
\Lambda_{\alpha,\beta}=\lambda_{\alpha, \beta} (\beta-\alpha) \chi_{S_{\overline u}} \qquad \text{a.e. in } A_{\overline t} \, .
$$
and hence, since $\lambda_{\alpha, \beta} (\beta-\alpha) \chi_{S_{\overline u}}=0$ a.e. in $A_{\overline t}\setminus S_{\overline u}$ and
$|A_{\overline t}\setminus S_{\overline u}|>0$, we obtain $\Lambda_{\alpha,\beta}=0$ and this is absurd.

\textbf{Step 4.} We complete the proof of the lemma. First of all, we observe that by Step 3, it is not restrictive, up to a set of zero measure, to assume that
$A_{\overline t}\setminus S_{\overline u}=\emptyset$ in such way that $A_{\overline t}\subseteq S_{\overline u}$ and, in turn,
\begin{equation} \label{eq:Su}
S_{\overline u}=\{(x,y)\in \Omega: \overline u^2(x,y)\leq \overline t\} \, .
\end{equation}
It remains to prove that $\overline \eta=\chi_{S_{\overline u}}$ a.e. in $\Omega$. Since $(\overline \eta,\overline u)$ and $(\chi_{S_{\overline u}},\overline u)$
are both optimal pairs we have

\begin{align*}
& \Delta^2 \overline u+ \lambda_{\alpha, \beta} (\beta-\alpha) \chi_{S_{\overline u}} \, \overline u=\Lambda_{\alpha, \beta} \overline u \quad
\text{a.e. in } \Omega \, , \\[7pt]
& \Delta^2 \overline u+ \lambda_{\alpha, \beta} (\beta-\alpha) \overline \eta \, \overline u=\Lambda_{\alpha, \beta} \overline u \quad
\text{a.e. in } \Omega \, ,
\end{align*}

thus implying that
\begin{equation*}
(\chi_{S_{\overline u}}-\overline \eta) \, \overline u=0 \qquad \text{a.e. in } \Omega \, .
\end{equation*}
It is easy to check that $\overline \eta=\chi_{S_{\overline u}}$ a.e. in $\{(x,y)\in \Omega:\overline u^2(x,y)\ge \overline t\}$ being $\overline t>0$. In order to prove that
$\overline \eta=\chi_{S_{\overline u}}$ a.e. in $\{(x,y)\in \Omega:\overline u^2(x,y)<\overline t\}$, we apply \eqref{eq:optimal} to $\overline u$, $\chi_{S_{\overline u}}$
and $\overline \eta$ observing that the inequality \eqref{eq:optimal} is an equality being $(\overline \eta,\overline u)$ and $(\chi_{S_{\overline u}},\overline u)$ both
optimal pairs. In particular we have that
$$
\int_{\{\overline u^2<\overline t\}} \overline u^{2} \, (\chi_{S_{\overline u}}-\overline\eta) \, dx\, dy
=\overline t \int_{\{ \overline u^2<\overline t\}} (\chi_{S_{\overline u}}-\overline\eta) \, dx\, dy \, ,
$$
which implies
\begin{equation} \label{eq:=0}
\int_{\{\overline u^2<\overline t\}} (\overline u^{2}-\overline t) \, (\chi_{S_{\overline u}}-\overline\eta) \, dx\, dy=0\,.
\end{equation}
But the function $(\overline u^{2}-\overline t) \, (\chi_{S_{\overline u}}-\overline\eta)\le 0$ in $\{(x,y)\in \Omega:\overline u^2(x,y)<\overline t\}$, as one can deduce by
\eqref{eq:Su}, and hence by \eqref{eq:=0} we conclude that $\chi_{S_{\overline u}}=\overline\eta$ a.e. in the same set.

We have so proved that $\chi_{S_{\overline u}}-\overline\eta=0$ a.e. in $\Omega$ and this completes the proof of the lemma.
\end{proof}

{\bf Proof of Theorem \ref{thm-exist-qual} completed.}

The existence of an optimal pair $(\overline{p},\overline u) \in P_{\alpha, \beta} \times H^2_*(\Omega)$ follows
from Lemma \ref{exist}. If we put $\overline \eta:=\frac{\beta-\overline p}{\beta-\alpha}$ by Lemma \ref{thm-1} we deduce that $(\overline \eta,\overline u)$ is an optimal pair for
$\Lambda_{\alpha, \beta}=\lambda_{\alpha, \beta}\, \beta$. Moreover by Lemma \ref{Lambda} we also have that $\overline \eta=\chi_{S_{\overline u}}$ a.e. in $\Omega$
with $S_{\overline u}=\{(x,y)\in \Omega:\overline u^2(x,y)\le \overline t\}$ and
$\overline t$ as in \eqref{eq:def-t}. Hence we conclude that
$$
{\overline p}=\beta-{\overline \eta}(\beta-\alpha)=\alpha \chi_{S_{\overline u}}+\beta \chi_{S_{\overline u}^c} \, .
$$

\subsection{Proof of Proposition \ref{comparison_esy}.}
We prove the two statements separately.
\par\medskip \par
\textbf{Proof of Proposition \ref{comparison_esy}-(i).}
\par

We know that the function $u_1(x,y)=\varphi_1(y) \sin x$ introduced in \eqref{u1} is an eigenfunction corresponding to the least eigenvalue of \eq{weight}
with $p\equiv 1$.  Furthermore,
$$
\min_{u\in H^2_*(\Omega)\setminus \{0\}} \frac{\|u\|_{H^2_*(\Omega)}^2}{\|u\|_{2}^2}=\frac{\|u_1\|_{H^2_*(\Omega)}^2}{\|u_1\|_{2}^2}=\mu_{1,1} \, .
$$
Now, by exploiting the fact that $\varphi_1$ is even and increasing in $(0,\ell)$ and $p=p(y)$ is even, we deduce that
\begin{align*}
&\int_{\Omega} (1-p(y))\,u_1^2(x,y)\,dx\,dy=2 \int_0^{\pi} \int_{0}^{\ell}(1-p(y))\varphi_1^2(y)\, \sin^2 x\,dx\,dy\\
&\leq 2 \varphi_1^2(z) \int_0^{\pi} \int_{0}^{z}(1-p(y))\, \sin^2 x\, dx\,dy+2 \varphi_1^2(z) \int_0^{\pi} \int_{z}^{\ell}(1-p(y))\, \sin^2 x\,dx\,dy\\
&= \varphi_1^2(z) \,\pi \, \int_{0}^{\ell} (1-p(y))\,dy= 0\,,
\end{align*}
where in the last step we have exploited the fact that $\int_{\Omega} p(y)\,dx\,dy=|\Omega|$, therefore $ \int_{0}^{\ell}p(y)\,dy =\ell$.
Hence,
$$
\int_{\Omega} u_1^2(x,y)\,dxdy \leq \int_{\Omega}  p(y)  \,u_1^2(x,y)\,dxdy \, .
$$
From the above inequality we infer
$$
\mu_{1,1}= \frac{\|u_1\|_{H^2_*(\Omega)}^2}{\|u_1\|_{2}^2}\geq   \min_{u\in H^2_*(\Omega)\setminus \{0\}} \frac{\|u\|_{H^2_*(\Omega)}^2}{\|\sqrt{p}u\|_{2}^2}=\lambda_1(p) \, ,
$$
and the proof of the statement follows.
\par
\medskip
\par
\textbf{Proof of Proposition \ref{comparison_esy}-(ii).} \par

The idea of the proof is similar to that applied to prove statement $(i)$. By exploiting the fact that $\sin(\pi-x)=\sin(x)$ and $p(\pi-x)=p(x)$ for all $x\in (0, \frac{\pi}{2})$, we deduce that
\begin{align*}
&\int_{\Omega} (1-p(x))\,u_1^2(x,y)\,dx\,dy =2  \int_{-\ell}^{\ell} \int_0^{\frac{\pi}{2}}(1-p(x))\varphi_1^2(y)\, \sin^2 x\,dx\,dy\\
&\leq 2  \sin^2 (s) \int_{-\ell}^{\ell} \int_0^{s} (1-p(x))\, \varphi_1^2(y)\,dx\, \,dy+2  \sin^2 (s)  \int_{-\ell}^{\ell} \int_s^{\frac{\pi}{2}}(1-p(x))\,\varphi_1^2(y) \,dx\,dy\\
&= 2 \sin^2 (s) \, \left( \int_{-\ell}^{\ell} \varphi_1^2(y)\,dy\right) \, \left(\int_{0}^{\frac{\pi}{2}}(1-p(x))\,dx\right)= 0\,,
\end{align*}
where in the last step we have exploited the assumption $\int_{\Omega} p(x)\,dx\, dy=|\Omega|$, hence $\int_{0}^{\frac{\pi}{2}}p(x)\,dx=\frac{\pi}{2}$. From the above inequality the proof follows as for statement $(i)$.

\section{Proof of Theorem \ref{min}, Theorem \ref{signTH} and Theorem \ref{corPPP}} \label{spectrum}
In this section we restrict the admissible weights to the family $\overline P_{\alpha,\beta}$ defined in Theorem \ref{min}. Clearly, $\int_{0}^{\ell}p\,dy=\ell$ for all $p\in \overline P_{\alpha,\beta}$. Furthermore, for $m$ positive integer, $\langle \cdot ,\cdot \rangle_m$ will denote the scalar product in $H^2(-\ell, \ell)$ defined in \eqref{mprodotto} with equivalent norm  $\lv \phi\rv_{m}^2  =(\phi,\phi)_m$.

Let $u$ be an eigenfunction of \eq{eigenweak}, its Fourier expansion reads
$$
u(x,y)=\sum_{m=1}^{+\infty} \varphi_m(y)\sin(mx)
$$
with $\varphi_m \in C^2([-\ell,\ell])$ since $u\in H^4(\Omega)$ (at least).
Inserting $u$ in \eq{eigenweak}, we get that, for every $m\geq 1$ fixed, $\varphi_m$ satisfies the equation
\begin{equation}\label{weight1dweak}
\langle \varphi,\phi \rangle_m=\lambda\int_{-\ell}^{\ell} p(y) \varphi \phi \,dy \quad \text{for all }\phi \in H^2(-\ell, \ell)\,
\end{equation}
which is the weak formulation of the problem \eqref{weight1d}. Notice that, by elliptic regularity, any solution $\varphi \in H^2(-\ell, \ell)$ of  \eqref{weight1d} , lies in $H^4(-\ell, \ell)\subset C^3([-\ell, \ell])$.
Hence, the boundary conditions in \eq{weight1d} are satisfied pointwise. Since the bilinear form $\langle\varphi,\phi\rangle_m$ is continuous
and coercive the
eigenvalues of problem \eq{weight1dweak} may be ordered in an increasing sequence of strictly positive
numbers diverging to $+\infty$ and the corresponding eigenfunctions form a complete
system in $H^2(-\ell, \ell)$. Whence, for what remarked so far, when $p=p(y)$ there is a one to one correspondence between eigenvalues of \eq{weight1dweak} and eigenvalues
of \eq{eigenweak}.
In particular, if we denote by $\lambda_1(p)$ the first eigenvalue of \eq{eigenweak} and by $\overline\lambda_1(p,m)$ the first eigenvalue of \eq{weight1dweak} with $m\geq 1$ fixed,
namely
$$
\lambda_1(p):= \min_{u\in H^2_*(\Omega)\setminus \{0\}} \frac{\|u\|_{H^2_*(\Omega)}^2}{\|\sqrt{p}u\|_{2}^2}\quad \text{and} \quad \overline\lambda_1(p,m)
:=\min_{\varphi\in H^2(-\ell, \ell)\setminus \{0\}}\frac{\lv\varphi\rv_{m}^2}{\|\sqrt{p} \varphi\|_{2}^2} \, ,
$$
it is natural to conjecture that
$$
\lambda_1(p)=\min_{m\geq 1}\left\{\overline\lambda_1(p,m)\right\}=\overline\lambda_1(p,1) \, .
$$
Unfortunately, for $p\in \overline P_{\alpha,\beta}$ fixed, due to the negative terms in the norm $\lv\, \cdot\rv_{m}$,
the monotonicity of $m\mapsto \overline\lambda_1(p,m)$ is not easy to detect and we do not have a proof of the above equality for general $p$;
in Section \ref{numerics} we give some suggestions through numerical experiments. Nevertheless, we have the following partial result

\begin{lemma}\label{comp1}
If $p\in \overline P_{\alpha,\beta}$
then
$$
\overline\lambda_1(p,m)\leq \mu_{m,1}<m^4 \, ,
$$
where the $\mu_{m,1}$ are the numbers defined in Proposition \ref{eigenvalue}-(i).\par
\noindent
If furthermore $\beta\leq 16(1-\sigma^2) $, then
\begin{equation}\label{monot}
\overline\lambda_1(p,m) \geq \overline\lambda_1(p,1) \quad \text{ for all } m\geq 2 \,.
\end{equation}

\end{lemma}

\begin{proof}
Let
$$\varphi_m(y):=\left[\big[\mu_{m,1}^{1/2}-(1-\sigma)m^2\big]\, \tfrac{\cosh\Big(y\sqrt{m^2+\mu_{m,1}^{1/2}}\Big)}{\cosh\Big(\ell\sqrt{m^2+\mu_{m,1}^{1/2}}\Big)}+
\big[\mu_{m,1}^{1/2}+(1-\sigma)m^2\big]\, \tfrac{\cosh\Big(y\sqrt{m^2-\mu_{m,1}^{1/2}}\Big)}{\cosh\Big(\ell\sqrt{m^2-\mu_{m,1}^{1/2}}\Big)}\right]\,,$$
From Proposition \ref{eigenvalue} it is readily deduced that $\varphi_m(y)$ is an eigenfunction corresponding to the least eigenvalue of \eq{weight1dweak} with $p\equiv 1$ and $m\geq 1$ fixed (otherwise we will find an eigenvalue of \eq{weight} not included in those listed in Proposition \ref{eigenvalue}).
Furthermore,
$$
\min_{\varphi\in H^2(-\ell, \ell)\setminus \{0\}}\frac{\lv\varphi\rv_{m}^2}{\|\varphi\|_{2}^2}
=\frac{\lv\varphi_m\rv_{m}^2}{\|\varphi_m\|_{2}^2}=\mu_{m,1} \, .
$$
Now, by exploiting the fact that $\varphi_m$ is even and increasing in $(0,\ell)$, the first part of the proof follows with the same
argument of Proposition \ref{comparison_esy}-(i), hence we omit it.
\par
Next we turn to the second estimate. Let $\varphi_{m,p}(y)$ be an eigenfunction corresponding to the least eigenvalue of \eq{weight1dweak}, with $m\geq  2$
fixed and with $p\in \overline P_{\alpha,\beta}$ satisfying the assumption of Lemma \ref{comp1}. In particular, $\varphi_{m,1}=\varphi_{m}$, with $\varphi_{m}$ as given above.
Since $p(y)\leq (1-\sigma^2)2^4$ for every $y\in(-\ell, \ell)$, we get
$$
\overline\lambda_1(p,m)=\frac{\lv\varphi_{m,p}\rv_{m}^2}{\|\sqrt{p} \varphi_{m,p}\|_{2}^2}\geq \frac{1}{(1-\sigma^2)m^4}
\frac{\lv\varphi_{m,p}\rv_{m}^2}{\| \varphi_{m,p}\|_{2}^2}\geq \frac{\mu_{m,1}}{(1-\sigma^2)m^4} \, .
$$
Then, the thesis follows by recalling that, from Proposition \ref{eigenvalue}-(i), $\mu_{m,1}\in ((1-\sigma^2) m^4, m^4)$ for every $m\geq 1$
and from the first part of the proof $\overline\lambda_1(p,1)<1$.
\end{proof}

Hence, under the assumptions of Lemma \ref{comp1}, we have
$$\lambda_1(p)=\overline\lambda_1(p,1)\leq \mu_{1,1}=\lambda_1(1)\,.$$
In particular, the weights considered in Lemma \ref{comp1} prove to be effective in lowering the first frequency of \eq{weight}, which is one of the main goal of
the present analysis. In the following we refine the result by carrying on a more deeper analysis. First we note that, from above, if $\varphi_{1,p}(y)$ is an eigenfunction
of $\overline\lambda_1(p,1)$, then $u_{1,p}(x,y):=\varphi_{1,p}(y)\sin (x)$ is an eigenfunction of $\lambda_1(p)$. Therefore, $\varphi_{1,p}(y)$ and $u_{1,p}(x,y)$
have the same sign.\par

We now discuss the sign of $\varphi_{1,p}(y)$ and the simplicity of $\lambda_1(p)$ in

\begin{lemma}\label{sign}
 Let $m\geq 1$  integer fixed and let $p\in \overline P_{\alpha,\beta}$. Then, the first eigenvalue $\overline\lambda_1(p,m)$ of problem \eq{weight1dweak}
 is simple and the first eigenfunction $\varphi_{m,p}(y)$ is of one sign in $[-\ell,\ell]$. \par \noindent Furthermore, if $\beta\leq 16(1-\sigma^2) $
 the same conclusion holds for the first eigenvalue $\lambda_1(p)$ of \eq{eigenweak}, namely it is simple and the corresponding eigenfunction is given by
 $u_{1,p}(x,y)=\varphi_{1,p}(y)\sin (x)$, hence of one sign in $\Omega$.
\end{lemma}

\begin{proof}
	We apply the decomposition with respect to dual cones technique,
	see \cite[Chapter 3]{book} suitably combined with Theorem
	\ref{ppp}.
Let $\mathcal{K}=\{\varphi\in H^2(-\ell,\ell):\varphi\geq 0 \text{ in }(-\ell,
\ell)\}$  and let $\mathcal{K}^*$ be its dual cone,
namely
$$
\mathcal{K}^*:=\{\psi\in H^2(-\ell,\ell):\langle \psi,\phi\rangle_m\leq 0\quad \text{ for
all }\phi \in \mathcal{K}\} \,,
$$
where $\langle \cdot,\cdot\rangle_m$ is defined in \eqref{mprodotto}.
Then, for any $\varphi\in H^2(-\ell,\ell)$
there exists a unique $(\chi, \psi)\in \mathcal{K} \times
\mathcal{K}^*$ such that
$$
\varphi=\chi+\psi\,, \quad \langle \chi, \psi\rangle_m=0 \,.
$$
Now we know that
$$
\overline\lambda_1(p,m)=\min_{\varphi\in H^2(-\ell, \ell)\setminus \{0\}}\frac{\lv\varphi\rv_{m}^2}{\|\sqrt{p} \varphi\|_{2}^2}
=\frac{\lv\varphi_{m,p}\rv_{m}^2}{\|\sqrt{p} \varphi_{m,p}\|_{2}^2} \, .
$$
For contradiction, assume that $\varphi_{m,p}$ changes sign. Then,
we may decompose $\varphi_{m,p}=\chi_{m,p}+\psi_{m,p}$ with
$\chi_{m,p} \in \mathcal{K}\setminus\{0\}$ and $\psi_{m,p} \in
\mathcal{K}^* \setminus\{0\}$.

In the remaining part of this proof we need some results on a positivity preserving property
which is treated in Section \ref{pppproof}.

From Corollary \ref{pppweak}, we deduce that $\psi_{m,p}<0$ in $(-\ell, \ell)$.
Then, replacing $\varphi_{m,p}$ with $\chi_{m,p}-\psi_{m,p}$, exploiting the fact that
$\chi_{m,p}-\psi_{m,p}>\varphi_{m,p}$ in $(-\ell, \ell)$ and
the orthogonality of $\chi_{m,p}$ and $\psi_{m,p}$ in $H^2(-\ell,\ell)$, we infer

$$
\frac{\lv\chi_{m,p}-\psi_{m,p}\rv_{m}^2}{\|\sqrt{p}
(\chi_{m,p}-\psi_{m,p})\|_{2}^2}
<\frac{\lv\varphi_{m,p}\rv_{m}^2}{\|\sqrt{p}
\varphi_{m,p}\|_{2}^2}\,,$$ a contradiction. Hence
$\varphi_{m,p}\geq 0$ in $(-\ell, \ell)$ and since $\varphi_{m,p}$
solves \eq{weight1dweak}, by Theorem \ref{ppp} with
$f=\overline\lambda_1(p,m)\, p(y)\, \varphi_{m,p}$, we conclude
that $\varphi_{m,p}> 0$ in $[-\ell, \ell]$.\par The
simplicity follows by noting that if $\varphi_{m,p}$ and $\bar
\varphi_{m,p}$ are two  linearly independent  positive minimizers, then
$\varphi_{m,p}+t \bar \varphi_{m,p}$ is a sign-changing minimizer
for some $t<0$ suitably chosen, a contradiction.
\end{proof}

\par

Next we focus on the sign of $\varphi_{1,p}'(y)$ and we prove

\begin{lemma}\label{convex}
If $p\in \overline P_{\alpha,\beta}$ is such that $\beta<1/\mu_{1,1}$
and if $\varphi_{1,p}$ is a positive eigenfunction of \eq{weight1dweak} with $m=1$ corresponding to the first eigenvalue $\overline \lambda_1(p,1)$, then
$\varphi_{1,p}$ is increasing in $(0,\ell)$.
\end{lemma}

\begin{proof}
For shortness we will write $\varphi_1$ instead of $\varphi_{1,p}$. Since $p$ is even, being $\varphi_1$ positive, we infer that it is an even function.
Hence, since $\varphi_1\in C^3([-\ell, \ell])$ it satisfies $\varphi_1'(0)=0=\varphi_1'''(0)$.  \par
 If $p$ is continuous, then $\varphi_{1}\in C^4([-\ell, \ell])$ and it satisfies the equation in \eq{weight1d} pointwise.
 We recall that the boundary conditions in \eq{weight1d} are satisfied pointwise also when $p$ is not continuous.
 Since $\varphi_1$ is positive, $\beta<1/\mu_{1,1}$ and, by Lemma \ref{comp1}, we know that $\overline \lambda_1(p,1)\leq \mu_{1,1}$, from the equation we infer
\begin{equation}\label{crucial}
\varphi_{1}''''(y)-2\varphi_{1}''(y)=( \overline \lambda_1(p,1)  p(y)-1)\varphi_{1}(y)\leq (\mu_{1,1} p(y)-1)\varphi_{1}(y)<0\quad \text{ in }(-\ell, \ell)
\end{equation}
If $p$ is not continuous, since only a finite number of points of
jump discontinuity are allowed in $(-\ell,\ell)$, say
$\{\tau_j\}_{j=1}^r$ for some integer $r$, the above inequality
holds in each interval $(\tau_j, \tau_{j+1})$. Furthermore, for
any $j=1,...,r$, the right and left fourth order derivative at
$\tau_j$ exists and they are given by
$(\varphi_1)''''_{\pm}(\tau_j)=\lim_{y\rightarrow
\tau_j^{\pm}}\varphi_{1}''''(y)$.\par

First we show that
\begin{equation}\label{der-sign}
\varphi_{1}' \text{ never vanishes in }(0, \ell)\,.
\end{equation}

By contradiction, let $y_1\in (0,\ell)$ be such that $\varphi_{1}'(y_1)=0$. Since $\varphi_{1}'(0)=0$ and $\varphi_1\in C^3([-\ell, \ell])$,
there exists $y_0\in (0, y_1)$ such that $\varphi_{1}''(y_0)=0$ and, by \eq{crucial}, $(\varphi_1)''''_{+}(y_0)<0$. Next the following two cases may occur.

$\bullet $ CASE 1: $\varphi_{1}'''(y_0)\leq 0$. From above, $\varphi_{1}'''$ is negative and, in turn, also $\varphi_{1}''$ is negative in a right neighborhood of $y_0$.
Since the boundary conditions in \eq{weight1d} yield $\varphi_{1}''(\ell)=\sigma \varphi_{1}(\ell)>0$, we infer that there exists $y_2>y_0$ such that
$\varphi_{1}''(y_2)=0$, $\varphi_{1}'''(y_2) \ge  0$ and $\varphi_{1}''(y)\leq 0$ in $(y_0,y_2)$. Whence, by \eq{crucial}, $\varphi_{1}''''(y)< 0$ in $(y_0,y_2)$
or in each of the subintervals $(\tau_j, \tau_{j+1})$ contained in $(y_0,y_2)$. Since $\varphi_{1}'''$ is continuous in $[y_0,y_2]$, in any case,
we have that it is strictly decreasing in $[y_0,y_2]$, hence $\varphi_{1}'''(y)< 0$ in  $(y_0,y_2]$ in contradiction with $\varphi_{1}'''(y_2)\ge 0$.

\medskip

$\bullet $ CASE 2: $\varphi_{1}'''(y_0)>0$. We distinguish two further cases.

\medskip

CASE 2a: $\varphi_{1}''(0)\leq 0$. By \eq{crucial}, $(\varphi_1)''''_{+}(0) <  0$, hence $\varphi_{1}'''(y)<0$ in a right neighborhood of $0$.
Then, since $\varphi_{1}'''(y_0)>0$, there exists $y_3\in (0,y_0)$ such that $\varphi_{1}'''(y)<0$ in $(0,y_3)$ and $\varphi_{1}'''(y_3)=0$. In turn, $\varphi_{1}''< 0$ in $(0,y_3)$ and by \eq{crucial} $\varphi_{1}''''(y)<0$ in $(0,y_3)$ (or in each of the subintervals $(\tau_j, \tau_{j+1})$ contained in $(y_0,y_3)$). Since $\varphi_{1}'''$ is continuous this implies that it is strictly decreasing in $[0,y_3]$. Since $\varphi_{1}'''(0)=0$, we infer $\varphi_{1}'''(y_3)<0$, a contradiction.

CASE 2b: $\varphi_{1}''(0)> 0$. From $\varphi_{1}'''(y_0)>0$ and $\varphi_{1}''(y_0)=0$ we infer that $\varphi_{1}''$ is negative in a left neighborhood of $y_0$. Then,  since  $\varphi_{1}''(0)> 0$, there exists $y_4\in (0, y_0)$ such that $\varphi_{1}''(y) >  0$ in $(0,y_4)$ and $\varphi_{1}''(y_4)=0$. In turn, recalling that $\varphi_{1}''(y_0)=0$, there exists $y_5\in (y_4,y_0)$ such that $\varphi_{1}'''(y_5)= 0$ and, by \eq{crucial}, we infer that $\varphi_{1}'''(y)< 0$ in $(y_5,y_0)$, in contradiction with $\varphi_{1}'''(y_0)>0$.

\medskip

Next we come back to the proof of the statement. By \eq{der-sign} we know that either $\varphi_{1}'(y)<0$ in $(0, \ell)$ or $\varphi_{1}'(y)>0$ in $(0, \ell)$.

Assume that $\varphi_{1}'(y)<0$ in $(0, \ell)$, then $\varphi_{1}''(0)\leq 0$. Indeed, if $\varphi_{1}''(0)>0$, since $\varphi_{1}'(0)=0$, then $\varphi_{1}'$ is positive in a right neighborhood of $0$, a contradiction. From $\varphi_{1}''(0)\leq 0$, together with \eq{crucial} and $\varphi_{1}'''(0)=0$, it follows that $\varphi_{1}'''$ is negative in a right neighborhood of $0$ and, in turn, also $\varphi_{1}''$ is negative in a right neighborhood of $0$. Since, from the boundary conditions $\varphi_{1}''(\ell)=\sigma \varphi_{1}(\ell)>0$, we deduce that there exists $\overline y\in(0, \ell)$ such that
$\varphi_{1}''(\overline y)=0$, $\varphi_{1}'''(\overline y) \ge  0$ and $\varphi_{1}''(y)\leq 0$ in $(0,\overline y)$\,. But then, from \eq{crucial}, $\varphi_{1}'''$ is strictly decreasing in $[0,\overline y]$ and,  recalling that $\varphi_1'''(0)=0$  we reach a contradiction.
\end{proof}

All the above statements yield the proof of Theorem \ref{min}.

\medskip

\textbf{Proof of Theorem \ref{min} completed.}

The key point is to note that, by Lemma \ref{convex}, we have

\begin{equation}\label{claim}
p\in \overline P_{\alpha,\beta} \quad \Rightarrow \quad \varphi_{1,p} \text{ increasing in } (0, \ell).
\end{equation}

Indeed, since by \eq{claim} $\varphi_{1,p_2}$ is increasing in $(0, \ell)$, to prove $(i)$ we may argue as in the proof of the first part of Lemma \ref{comp1} with $\varphi_{1,p_2}$ instead of $\varphi_{m}$. In particular, we readily infer that $\overline \lambda_1(p_1,1)\leq \overline \lambda_1(p_2,1)$ and since, from Lemma \ref{comp1}, $\overline \lambda_1(p,1)=\lambda_1(p)$ for all $p\in  \overline P_{\alpha,\beta}$, the proof of $(i)$ follows.\par
Next we prove $(ii)$. Set $\overline y:=\frac{\ell(\beta-1)}{\beta-\alpha}$, for every   $p\in P_{\alpha,\beta}$  there holds
$$p(y)\geq \overline p(y) \quad \text{in } [0,\overline y]\quad \text{ and }\quad p(y)\leq \overline p(y) \quad \text{in } [\overline y,\ell)\,.$$
Then, we may argue again as in the proof of the first part of Lemma \ref{comp1} with $\varphi_{1,p}$ instead of $\varphi_{m}$ and conclude
$$\overline\lambda_1( p,1)\geq\overline\lambda_1(\overline p,1)\,.$$
Once more, from Lemma \ref{comp1}, $\overline \lambda_1(p,1)=\lambda_1(p)$ for all $p\in  \overline P_{\alpha,\beta} $ and the statement of
Theorem  \ref{min} follows.
\par\medskip\par
\textbf{Proof of Theorem \ref{signTH}.}\par
The proof readily follows by combining the statements of Lemma \ref{sign} and Lemma \ref{convex}.

\par\medskip\par
\textbf{Proof of Theorem \ref{corPPP}.}\par
The proof readily follows as a corollary of Theorem \ref{ppp} by exploiting the same separation of variables performed in the Proof of Theorem \ref{min}.

\par\bigskip\noindent
\textbf{Acknowledgments.} The authors are grateful to the anonymous referee whose relevant comments and suggestions helped in improving the exposition of the preliminary form of the manuscript. The first three authors are members of the Gruppo Nazionale per l'Analisi Matematica, la Probabilit\`a
e le loro Applicazioni (GNAMPA) of the Istituto Nazionale di Alta Matematica (INdAM) and are partially supported by the INDAM-GNAMPA 2019 grant ``Analisi spettrale per operatori ellittici con condizioni di Steklov o parzialmente incernierate'' and by the PRIN project ``Direct and inverse problems for partial differential equations: theoretical aspects and applications'' (Italy). The third author was partially supported by the INDAM-GNAMPA 2018 grant ``Formula di monotonia e applicazioni: problemi frazionari e stabilit\`a spettrale rispetto a perturbazioni del dominio'' and by the research project ``Metodi analitici, numerici e di simulazione per lo studio di equazioni differenziali a derivate parziali e applicazioni” Progetto di Ateneo 2016, Universit\`a del Piemonte Orientale ``Amedeo Avogadro''. The research of the fourth author is partially supported by an INSPIRE faculty fellowship (India).

\end{document}